\documentclass[11pt,a4paper]{amsart}
\usepackage[utf8]{inputenc}

\usepackage[T1]{fontenc}
\usepackage[english]{babel}

\usepackage{sidecap}

\usepackage{verbatim}
\usepackage{lmodern}
\usepackage{amsmath}
\usepackage{amssymb} 
\usepackage{amsthm} 
\usepackage{thmtools}
\usepackage{enumitem}
\usepackage[margin=1.1in]{geometry}
\usepackage{graphicx}
\usepackage{hyperref}
\usepackage{mathtools}
\usepackage{indentfirst}
\usepackage{tabularx}
\usepackage{bbm}
\usepackage[capitalise]{cleveref}
\usepackage{stmaryrd}

\usepackage{tikz} %pictures
\usepackage{tikz-cd} %diagrams
\usetikzlibrary{arrows,calc,decorations.markings}
\usepackage{xparse}

\usepackage{blindtext}

\makeatletter
\newcommand{\@bbify}[1]{
  \ifcsname b#1\endcsname
  \message{WARNING: Overwriting b#1 with blackboard letter!}
  \fi
  \expandafter\edef\csname b#1\endcsname
  {\noexpand\ensuremath{\noexpand\mathbb #1}\noexpand\xspace}}
\newcommand{\@calify}[1]{
  \ifcsname c#1\endcsname
  \message{WARNING: Overwriting c#1 with calligraphic letter!}
  \fi 
  \expandafter\edef\csname c#1\endcsname
  {\noexpand\ensuremath{\noexpand\mathcal #1}\noexpand\xspace}}
\newcommand{\@bfify}[1]{
  \ifcsname bf#1\endcsname
  \message{WARNING: Overwriting bf#1 with bold letter!}
  \fi
  \expandafter\edef\csname bf#1\endcsname
  {\noexpand\ensuremath{\noexpand\mathbf #1}\noexpand\xspace}}
\newcommand{\@sfify}[1]{
  \ifcsname s#1\endcsname
  \message{WARNING: Overwriting s#1 with sf letter!}
  \fi
  \expandafter\edef\csname s#1\endcsname
  {\noexpand\ensuremath{\noexpand\mathsf #1}\noexpand\xspace}}
\newcounter{@letter}\stepcounter{@letter}
\loop\@bbify{\Alph{@letter}}\@calify{\Alph{@letter}}\@bfify{\Alph{@letter}}\@sfify{\Alph{@letter}}
\ifnum\the@letter<26\stepcounter{@letter}\repeat
\makeatother

\newenvironment{tz}{\begin{center}\begin{tikzpicture}}{\end{tikzpicture}\end{center}}

\tikzstyle{d}=[double distance=.3ex]

\NewDocumentCommand{\punctuation}{ m m O{5pt} }{\node at ($(#1.east)-(0,#3)$) {#2};}

\tikzset{%
node distance=1.5cm, la/.style={scale=0.8}, rr/.style={xshift=1.5cm},
space/.style={xshift=.5cm}, over/.style={auto=false,fill=white,inner sep=1.5pt, minimum size=0, outer sep=0},
    symbol/.style={%
        draw=none,
        every to/.append style={%
            edge node={node [sloped, allow upside down, auto=false]{$#1$}}},
            
    }, 
    pro/.style={postaction={decorate,decoration={markings,
      mark=at position .5 with {
        \draw[-]
        (0,{-1.25ex/2}) -- (0,{1.25ex/2});
      }},
      inner sep=.9ex,
      }},
  n/.style={double equal sign distance, -implies}, t/.style={double distance=2.5pt, -implies, postaction={draw,-}},
}

\newcommand{\varrow}{\mathrel{\begin{tikzpicture}
    [baseline=(current bounding box.south)]
    \draw[pro, ->] (0,0) -- (.4,0);
  \end{tikzpicture}}}

\def\cellslide{0.5}
\def\celllength{.2cm}

\NewDocumentCommand{\cell}{ O{} O{n} O{\cellslide} O{\celllength} m m m }{
  \coordinate (mid) at ($({#5})!{#3}!({#6})$);
  \coordinate (start) at ($(mid)!{#4}!({#5})$);
  \coordinate (end) at ($(mid)!{#4}!({#6})$);
  \draw[#2] (start) to node
  [inner sep=6pt,outer sep=0,minimum size=0,#1]{{#7}} (end);
}

\newcommand{\Nh}{N^h}
\newcommand{\ch}{c^h}
\newcommand{\Sd}{\mathrm{Sd}}
\newcommand{\Ex}{\mathrm{Ex}}
\newcommand{\op}{{\mathrm{op}}}
\newcommand{\colim}{\mathrm{colim}}
\newcommand{\hocolim}{\mathrm{hocolim}}
\newcommand{\id}{{\mathrm{id}}}
\newcommand{\Hom}{{\mathrm{Hom}}}
\newcommand{\Cat}{\mathsf{Cat}}
\newcommand{\DblCat}{\mathsf{DblCat}}
\newcommand{\Set}{\mathsf{Set}}
\newcommand{\sSet}{\mathsf{sSet}}
\newcommand{\Simp}{\mathsf{\Delta}}
\newcommand{\Simpop}{\Simp^{\op}}
\newcommand{\Sp}{\mathsf{Spaces}} 

\newcommand{\proj}{\mathrm{proj}}

\newcommand{\Seg}{\mathrm{Seg}}
\newcommand{\CSS}{\mathrm{CSS}}
\newcommand{\Gpd}{\mathrm{Gpd}}
\newcommand{\Grp}{\mathrm{Grp}}
\newcommand{\Kan}{\mathrm{Kan}}
\newcommand{\Thom}{\mathrm{Thom}}
\newcommand{\projThom}{\Cat^{\Simpop}_\proj}
\newcommand{\projKan}{\sSet^{\Simpop}_\proj}

\newcommand{\MSdiag}{\sSet^{\Simpop}_{\mathrm{diag}}}
\newcommand{\Ob}{\mathrm{Ob}}
\newcommand{\Mor}{\mathrm{Mor}}
\newcommand{\cst}{\mathrm{cst}}
\newcommand{\diag}{\mathrm{diag}}
\newcommand{\pushout}[1]{\node at ($({#1})-(10pt,-10pt)$) {$\ulcorner$};}
\newlist{rome}{enumerate}{7}
\setlist[rome]{label=(\roman*),leftmargin=0.85cm}

\newtheorem{theorem}{Theorem}[section]
\newtheorem{cor}[theorem]{Corollary}
\newtheorem{prop}[theorem]{Proposition}
\newtheorem{lemma}[theorem]{Lemma}
\declaretheorem[name=Theorem,numbered=yes]{theoremA}

\theoremstyle{definition}
\newtheorem{defn}[theorem]{Definition}
\newtheorem{ex}[theorem]{Example}
\newtheorem{notation}[theorem]{Notation}

\theoremstyle{remark}
\newtheorem{rem}[theorem]{Remark}

\declaretheoremstyle[
spaceabove=\topsep, spacebelow=\topsep,
headfont=\normalfont\bfseries,
notefont=\bfseries, notebraces={}{},
bodyfont=\normalfont,
postheadspace=0.5em,
name={\ignorespaces},
numbered=no,
headpunct=.]
{mystyle}

\crefname{theorem}{Theorem}{Theorems}
\crefname{cor}{Corollary}{Corollaries}
\crefname{prop}{Proposition}{Propositions}
\crefname{lemme}{Lemma}{Lemmas}

\crefname{defn}{Definition}{Definitions}
\crefname{ex}{Example}{Examples}
\crefname{notation}{Notation}{Notations}
\crefname{descr}{Description}{Descriptions}
\crefname{constr}{Construction}{Constructions}

\crefname{rem}{Remark}{Remarks}

\title{Double categorical model of $(\infty,1)$-categories}

\author{Léonard Guetta}
\address{Utrecht Geometry Center, Universiteit Utrecht, Utrecht, The Netherlands}
\email{l.s.guetta@uu.nl}

\author{Lyne Moser}
\address{Fakultät für Mathematik, Universität Regensburg, Regensburg, Germany}
\email{lyne.moser@ur.de}

\begin{document}

\begin{abstract}
    Building on work by Fiore--Pronk--Paoli, we construct four model structures on the category of double categories, each modeling one of the following: simplicial spaces, Segal spaces, $(\infty,1)$-categories, and $\infty$-groupoids. Additionally, we provide an explicit formula for computing homotopy colimits in these models using the Grothendieck construction. We expect the model of double categories for $(\infty,1)$-categories to play a similar role than that of the model of categories for spaces or $\infty$-groupoids in Grothendieck's study of the homotopy theory of spaces.
\end{abstract}

\maketitle

\setcounter{tocdepth}{1}
\tableofcontents

\section{Introduction}

Over the last two decades, the theory of higher categories, particularly $(\infty,1)$-categories, has become increasingly prominent in modern mathematics. While ordinary categories have sets of objects and of morphisms, which compose in an associative and unital way, $(\infty,1)$-categories provide homotopical versions of such structures. Sets of morphisms are replaced by \emph{spaces} (e.g.~topological spaces), which encode not only morphisms between objects but also higher morphisms: $2$-morphisms (homotopies), $3$-morphisms (homotopies between homotopies), and so on; in particular, all $k$-morphisms are invertible for $k>1$. Composition remains associative and unital, but only up to higher invertible morphisms. This relaxed structure accommodates a broader range of examples, positioning $(\infty,1)$-categories as a powerful language to formulate and demonstrate results in various fields, such as derived algebraic geometry, $K$-theory, and topological quantum field theory. 

Transitioning from intuition to a rigorous definition of $(\infty,1)$-categories is challenging due to the vast amount of data required to encode all higher coherences. A classical approach is to present their homotopy theory via a Quillen model structure on a category. In a model category, the key players are the \emph{weak equivalences}, which identify when two objects are ``essentially the same'', and the \emph{fibrant} objects, which exhibit the desired homotopical properties. For instance, one can study the homotopy theory of topological spaces up to weak homotopy equivalences, which capture the homotopy theory of \emph{spaces}. While all models are theoretically equivalent, some are better suited to specific goals, making it valuable to explore new models with different features.

One widely used model for $(\infty,1)$-categories is given by Rezk's complete Segal spaces \cite{Rezk}. To explain their definition, recall an alternative perspective on ordinary categories: the category of categories is equivalent to a full subcategory of the category $\Set^{\Simpop}$ of functors $\Simpop\to \Set$, with 
$\Simp$ the category of finite ordered sets $[m]\coloneqq \{0<1<\ldots<m\}$ and order-preserving maps. Categories correspond to those functors 
$\sC\colon \Simpop\to \Set$ satisfying the Segal conditions: for every $m\geq 2$, the Segal map 
\[
\sC_m \to \sC_1\times_{\sC_0}\cdots\times_{\sC_0}\sC_1
\] 
is a bijection. Here, we can think of the sets $\sC_0$, $\sC_1$, and $\sC_m$ as the sets of objects, morphisms, and $m$ composable morphisms, and composition, associativity, and unitality are encoded by the structure of $\Simp$ and the Segal conditions.

The homotopical generalization is then obtained by replacing sets, pullbacks, and bijections with the appropriate space-theoretic analogues. A \emph{Segal space} is then a functor $X \colon \Simpop \to \Sp$, where $\Sp$ denotes a model of the homotopy theory of spaces (e.g.~topological spaces), such that, for every $m \geq 2$, the Segal map
\[
X_m \to X_1\times^h_{X_0}\cdots\times^h_{X_0}X_1,
\]
is a weak homotopy equivalence of spaces. Here the pullbacks are replaced by \emph{homotopy pullbacks}. Such a structure has spaces $X_0$ of objects and $X_1$ of morphisms, and composition and higher coherences are encoded by the structure of $\Simp$ and the Segal conditions. However, Segal spaces do not fully model $(\infty,1)$-categories, as they a priori have two different underlying $\infty$-groupoids: the space 
$X_0$ of objects and the subspace of $X_1$ spanned by the invertible morphisms. A \emph{complete} Segal space is defined as a Segal space $X$ such that these two underlying $\infty$-groupoids are weakly homotopy equivalent.

In order to obtain a model of $(\infty,1)$-categories using complete Segal spaces, we consider the Kan--Quillen model structure on the category $\sSet\coloneqq \Set^{\Simpop}$ of simplicial sets, a model for the homotopy theory of spaces which has favorable combinatorial properties. Then Rezk \cite{Rezk} shows that the complete Segal spaces are the fibrant objects of a model structure on the category $\sSet^{\Simpop}$ of functors $\Simpop\to \sSet$, and this presents the homotopy theory of $(\infty,1)$-categories by a result of Joyal--Tierney \cite{JoyalTierney}. 

In \cite{Horel}, Horel proves a \emph{rigidification} result, using the category $\Cat(\sSet)$ of internal categories to $\sSet$. This category is equivalent to the full subcategory of $\sSet^{\Simpop}$ given by the functors $X \colon \Simpop \to \sSet$ such that, for every $m \geq 2$, the Segal map
\[
X_m \to X_1\times_{X_0}\cdots\times_{X_0}X_1
\]
is an isomorphism of simplicial sets. Although each value $X_m$ is a simplicial set, we consider here \emph{strict} pullbacks and isomorphisms, rather than their homotopical versions. While this might seem a priori insufficient to model $(\infty,1)$-categories, as this makes composition of $1$-morphisms \emph{strictly} associative and unital, Horel proves that there is a model structure on $\Cat(\sSet)$ such that the inclusion $\Cat(\sSet) \to \sSet^{\Simpop}$ induces an equivalence of homotopy theories with Rezk's model structure on $\sSet^{\Simpop}$ for complete Segal spaces. This provides another presentation of the homotopy theory of $(\infty,1)$-categories.

The main goal of this paper is to replace the model of spaces, given by $\sSet$ in Horel's result, with another one. Namely, we use Thomason's model structure on the category $\Cat$ of categories \cite{Thomason}, which is equivalent to the Kan--Quillen model structure on $\sSet$ and thus presents the homotopy theory of spaces. Replacing $\sSet$ with $\Cat$, we work with the category $\DblCat\coloneqq \Cat(\Cat)$ of internal categories to categories, usually referred to as \emph{double categories}. The main result of this paper, proven in \cref{sec:CSS}, can be stated as follows:

\begin{theoremA}\label{thmintro}
    There exists a model structure on the category $\DblCat$ of double categories such that the inclusion functor
\[
\bN\colon \DblCat\to \sSet^{\Simpop}
\]
induces an equivalence of homotopy theories with Rezk's model structure on $\sSet^{\Simpop}$ for complete Segal spaces. In particular, this model structure on $\DblCat$ presents the homotopy theory of $(\infty,1)$-categories.
\end{theoremA}

Although it might seem reasonable to expect such a result, it is not a trivial consequence of Horel's result. Indeed, the notion of internal category is \emph{not} homotopical, as it is defined using \emph{strict} pullbacks and isomorphisms. Hence, there is no guarantee that this notion behaves well when replacing one model of spaces with another. To illustrate this point, consider the following example. The category $\Grp(\sSet)$ of internal groups to simplicial sets is known to model pointed connected homotopy types \cite[Chapter V]{goerssjardine}, whereas the category $\Grp(\Cat)$ of internal groups to categories only models pointed connected homotopy $2$-types \cite{loday}.

We expect the model $\DblCat$ to play a role in the homotopy theory of $(\infty,1)$-categories analogous to the role of $\Cat$ in the homotopy theory of spaces, as conceptualized by Gro\-then\-dieck \cite{Grothendieck}. For instance, Gro\-then\-dieck's choice of $\Cat$ as the fundamental model for spaces enables the axiomatization of \emph{test categories}, a notion further studied by Maltsiniotis \cite{Maltsiniotis}. Test categories are categories whose associated presheaf category models the homotopy theory of spaces in a canonical way; this was studied through the lens of model categories by Cisinski \cite{Cisinski}. Additionally, Grothendieck \cite{Grothendieck} introduced the concept of \emph{basic localizers} in $\Cat$, which Cisinski \cite{Cisinski} further developed, showing that they provide a convenient framework to describe all (Bousfield) localizations of the homotopy theory of spaces. A crucial ingredient in both the theories of test categories and basic localizers is the explicit computation of homotopy colimits in $\Cat$ using the Grothendieck construction.

As a first compelling point supporting this expectation, we show that the model $\DblCat$ of the homotopy theory of $(\infty,1)$-categories admits a particularly convenient and explicit description of homotopy colimits.  This is achieved through the following construction. For a small category $\sJ$, the Grothendieck construction functor 
\[
\textstyle \int_\sJ \colon \Cat^\sJ \to \Cat,
\]
induces by postcomposition a functor $(\int_\sJ)_* \colon (\Cat^{\Simpop})^\sJ\cong (\Cat^\sJ)^{\Simpop} \to \Cat^{\Simpop}$, which in turn restricts to a functor $(\int_\sJ)_* \colon \DblCat^\sJ \to \DblCat$. We prove in \cref{sec:hocolimdblcat} the following result.

\begin{theoremA}
A model for the homotopy colimit functor is given by
\[ \textstyle (\int_\sJ)_* \colon \DblCat^\sJ \to \DblCat, \] 
where $\DblCat$ is endowed with the model structure from  \cref{thmintro} presenting the homotopy theory of $(\infty,1)$-categories.
\end{theoremA}

Similarly to the case of $\Cat$, the above functor preserves all weak equivalences, so it does not need to be derived. As a result, it offers a much more practical formula for homotopy colimits than the one provided by the abstract theory; see \cref{explicitGCdblcat,sec:application}.

Another potential application of the model structure on $\DblCat$ is that it provides a model of $(\infty,1)$-categories on a category whose objects are described with a finite amount of data and relations, while retaining favorable properties. To illustrate this, recall that a relative category consists of a category $\sC$ equipped with a wide subcategory $\sW \subseteq \sC$. Relative categories form models of $(\infty,1)$-categories, and many examples of $(\infty,1)$-categories in fact arise from relative categories. The model of double categories occupies an intermediate position between relative categories and complete Segal spaces
\[
\mathsf{CatRel} \hookrightarrow \DblCat \hookrightarrow \sSet^{\Simpop},
\]
while already enjoying desirable features such as an explicit description of homotopy colimits. This perspective may lead to intriguing new developments, which we aim to explore in future work.

\subsection*{Overview of the proof of \texorpdfstring{\cref{thmintro}}
{Theorem A}}

The proof of \cref{thmintro} goes in three steps. First, we consider the ``projective model structure'' $\DblCat_{\proj}$, whose existence is established by Fiore--Paoli--Pronk in \cite{FPP} and recalled in \cref{sec:MSdblcat}, and we prove that it is left proper in \cref{sec:leftproper}. This model structure is obtained as the right-induced model structure along a right adjoint
\[ (\Ex^2)_*\bN\colon \DblCat_{\proj}\to \sSet_{\proj}^{\Simpop}, \]
where $\sSet_{\proj}^{\Simpop}$ denotes the projective model structure on simplicial objects in the Kan--Quillen model structure on $\sSet$. Note that the functor is not the canonical inclusion of double categories into bisimplicial sets. Instead, it is modified by applying the extension functor~$\Ex$ on simplicial sets twice, akin to the approach of Thomason in constructing the model structure on $\Cat$. We then demonstrate that this functor is, in fact, a Quillen equivalence in \cref{sec:QE}.

Finally, once we know that we have a Quillen equivalence between $\DblCat_{\proj}$ and $\sSet^{\Simpop}_{\proj}$, \cref{thmintro} follows by taking the left Bousfield localizations of these model structures at the Segal and completeness maps. We summarize the different model structures on $\DblCat$ appearing in this paper in the following commutative diagram of model categories.
\begin{tz}
    \node[](1) {$\DblCat_\proj$};
    \node[right of=1,xshift=2cm](2) {$\DblCat_\Seg$}; 
    \node[right of=2,xshift=2cm](3) {$\DblCat_\CSS$}; 
    \node[right of=3,xshift=2cm](4) {$\DblCat_\Gpd$}; 

    \node[below of=1](1') {$\sSet^{\Simpop}_\proj$};
    \node[below of=2](2') {$\sSet^{\Simpop}_\Seg$};
    \node[below of=3](3') {$\sSet^{\Simpop}_\CSS$};
    \node[below of=4](4') {$\sSet^{\Simpop}_\Gpd$};

    \draw[->](1) to node[above,la]{$\id$} (2);
    \draw[->](2) to node[above,la]{$\id$} (3);
    \draw[->](3) to node[above,la]{$\id$} (4);
    \draw[->](1') to node[below,la]{$\id$} (2');
    \draw[->](2') to node[below,la]{$\id$} (3');
    \draw[->](3') to node[below,la]{$\id$} (4');
    \draw[->](1) to node[left,la]{$\simeq$} node[right,la]{$(\Ex^2)_*\bN$} (1');
    \draw[->](2) to node[left,la]{$\simeq$} node[right,la]{$(\Ex^2)_*\bN$} (2');
    \draw[->](3) to node[left,la]{$\simeq$} node[right,la]{$(\Ex^2)_*\bN$} (3');
    \draw[->](4) to node[left,la]{$\simeq$} node[right,la]{$(\Ex^2)_*\bN$} (4');
\end{tz}
The horizontal functors are localization functors, while the vertical ones are Quillen equivalences, and all the model structures on $\DblCat$ are in fact right-induced from the corresponding model structures on $\sSet^{\Simpop}$; see \cref{thm:RIMS}. From left to right, they provide models for the homotopy theory of simplicial spaces, Segal spaces (\cref{sec:Segal}), complete Segal spaces (\cref{sec:CSS}), and $\infty$-groupoids or spaces (\cref{sec:gpdloc}). We further compare in \cref{sec:GpdvsThom} our double categorical model of spaces with Fiore--Paoli's Thomason-like model structure on $\DblCat$ \cite{FP}, also modeling spaces. 

Furthermore, we show that, for each of the above localizations, the canonical inclusion $\bN\colon \DblCat\to \sSet^{\Simpop}$ also preserve and reflects weak equivalences, and induces an equivalence at the level of underlying $\infty$-categories. Hence we obtain a commutative diagram of $\infty$-categories
\begin{tz}
    \node[](1) {$(\DblCat_\proj)_\infty$};
    \node[right of=1,xshift=2cm](2) {$(\DblCat_\Seg)_\infty$}; 
    \node[right of=2,xshift=2cm](3) {$(\DblCat_\CSS)_\infty$}; 
    \node[right of=3,xshift=2cm](4) {$(\DblCat_\Gpd)_\infty$}; 

    \node[below of=1](1') {$(\sSet^{\Simpop}_\proj)_\infty$};
    \node[below of=2](2') {$(\sSet^{\Simpop}_\Seg)_\infty$};
    \node[below of=3](3') {$(\sSet^{\Simpop}_\CSS)_\infty$};
    \node[below of=4](4') {$(\sSet^{\Simpop}_\Gpd)_\infty$};

    \draw[->](1) to node[above,la]{$\mathrm{loc}$} (2);
    \draw[->](2) to node[above,la]{$\mathrm{loc}$} (3);
    \draw[->](3) to node[above,la]{$\mathrm{loc}$} (4);
    \draw[->](1') to node[below,la]{$\mathrm{loc}$} (2');
    \draw[->](2') to node[below,la]{$\mathrm{loc}$} (3');
    \draw[->](3') to node[below,la]{$\mathrm{loc}$} (4');
    \draw[->](1) to node[left,la]{$\simeq$} node[right,la]{$\bN_\infty$} (1');
    \draw[->](2) to node[left,la]{$\simeq$} node[right,la]{$\bN_\infty$} (2');
    \draw[->](3) to node[left,la]{$\simeq$} node[right,la]{$\bN_\infty$} (3');
    \draw[->](4) to node[left,la]{$\simeq$} node[right,la]{$\bN_\infty$} (4');
\end{tz}
where the horizontal functors are localizations functors and the vertical functors are equivalences of $\infty$-categories induced by the canonical inclusion $\bN$.

\subsection*{Acknowledgments}

We would like to thank Denis-Charles Cisinski, Joshua Lieber, and Nima Rasekh for helpful discussions related to the subject of this paper.

During the realization of this work, the second author was a member of the Collaborative Research Centre ``SFB 1085: Higher
Invariants'' funded by the Deutsche Forschungsgemeinschaft (DFG).

\section{Model categorical preliminaries}

In this section, we introduce the necessary model categorical background for the paper. In \cref{sec:indMS}, we recall two methods for inducing a model structure along an adjunction: the classical \emph{right-induction} and the \emph{fibrant-induction} introduced in \cite{GMSV}. Then, in \cref{sec:locMS}, we recall the notion of \emph{left Bousfield localization} of a model structure. In \cref{sec:indvslocMS}, we compare these two approaches---induction and localization---and prove in \cref{thm:RIMS} that a right-induced model structure along a Quillen equivalence remains right-induced after localizing on both sides. Finally, in \cref{sec:hocolimMS}, we recall some results on homotopy colimits in a model category.

\subsection{Induced model structures} \label{sec:indMS}

We recall methods for inducing model structures along adjunctions. For this, consider an adjunction
   \begin{tz}
\node[](A) {$\sM$};
\node[right of=A,xshift=.6cm](B) {$\sN$};
\punctuation{B}{,};
\draw[->] ($(A.east)+(0,5pt)$) to node[above,la]{$F$} ($(B.west)+(0,5pt)$);
\draw[->] ($(B.west)-(0,5pt)$) to node[below,la]{$U$} ($(A.east)-(0,5pt)$);
\node[la] at ($(A.east)!0.5!(B.west)$) {$\bot$};
\end{tz}
where $\sM$ is a model category. Classically, the model structure on $\sM$ can be transferred to the category $\sN$ via the right adjoint $U$, by pulling back the weak equivalences and fibrations. 

\begin{defn}
    The \textbf{right-induced model structure} on $\sN$, if it exists, is such that a map $f$ is a weak equivalence (resp.~fibration) in $\sN$ if and only if $Uf$ is a weak equivalence (resp.~fibration) in~$\sM$. 
\end{defn}

In \cite[\textsection 3]{GMSV}, the authors introduce a relaxed version of this transfer, where the trivial fibrations remain induced, but the fibrations and weak equivalences are now only transferred between fibrant objects.

\begin{defn}
The \textbf{fibrantly-induced model structure} on $\sN$, if it exists, is such that
    \begin{rome}
        \item a map $f$ is a trivial fibration in $\sN$ if and only if $Uf$ is a trivial fibration in $\sM$, 
        \item an object $Y$ is fibrant in $\sN$ if and only if $UY$ is fibrant in $\sM$, 
        \item a map $f$ between fibrant objects is a weak equivalence (resp.~fibration) in~$\sN$ if and only if $Uf$ is a weak equivalence (resp.~fibration) in $\sM$.
    \end{rome}
\end{defn}

\begin{rem}
    Since the fibrantly-induced and right-induced model structures have the same trivial fibrations and fibrant objects, if the right-induced model structure exists, then it coincides with the fibrantly-induced one. However, there are cases where the fibrantly-induced model structure exists, but the right-induced one does not; see \cite[\textsection 7]{GMSV}.
\end{rem}

\begin{rem}
    For either induced model structure, the adjunction $F\dashv U$ becomes a Quillen pair.
\end{rem}

We now give a useful criterion for recognizing when the fibrantly-induced model structure is in fact the right-induced one.

\begin{prop} \label{criterionfitori}
    Suppose that $\sN$ is locally presentable and $\sM$ is combinatorial. If the fibrantly-induced model structure on $\sN$ exists and the right adjoint $U$ preserves all weak equivalences, then the right-induced model structure on $\sN$ exists. 
\end{prop}

\begin{proof}
    By \cite[Corollary 3.3.4]{HKRS}, it suffices to show the acyclicity condition: for every map $f$ in $\sN$ that has the left lifting property against all maps $p$ in $\sN$ such that $Up$ is a fibration in $\sM$, then $Uf$ is a weak equivalence in $\sN$. Suppose that $f$ is a map in $\sN$ satisfying the left lifting property against all maps $p$ in $\sN$ such that $Up$ is a fibration in~$\sM$. Then it satisfies the left lifting property against all maps $p\colon X\to Y$ in $\sN$ such that $Up\colon UX\to UY$ is a fibration between fibrant objects in $\sM$. Hence, by \cite[Lemma E.2.13]{JoyalVolumeII}, the map $f$ is a trivial cofibration in $\sN$ and, in particular, it is a weak equivalence. As $U$ preserves weak equivalences by assumption, then $Uf$ is a weak equivalence in $\sM$.
\end{proof}

\subsection{Localized model structures} \label{sec:locMS}

We recall techniques for localizing a model structure at a set of morphisms. Consider a set $S$ of morphisms in a model category $\sM$. We start by recalling from \cite[Definition 2.1.2]{Hirschhorn} the definition of the \emph{left Bousfield localization} of the model structure~$\sM$ at the set~$S$.

\begin{notation}
    Given objects $X,Y\in \sM$, we denote by $\bR\Hom_\sM(X,Y)$ the derived hom of $\sM$ from $X$ to $Y$. Recall that, if $Y$ is fibrant in $\sM$, the derived hom can be modeled by the simplicial set $\Hom_\sM(X_\bullet,Y)$ obtained by taking a cosimplicial resolution $X_\bullet$ of $X$. See \cite[\textsection 17.1]{Hirschhorn} for more details.
\end{notation}

\begin{defn}
    We say that 
    \begin{itemize}[leftmargin=0.6cm]
        \item an object $X\in\sM$ is \textbf{$S$-local} if it is fibrant in $\sM$ and, for every map $s\colon A\to B$ in $S$, the induced map on derived homs 
        \[ s^*\colon \bR\Hom_\sM(B,X)\to \bR\Hom_\sM(A,X) \]
        is a weak homotopy equivalence,
        \item a map $f\colon A\to B$ in $\sM$ is an \textbf{$S$-local equivalence} if, for every $S$-local object $X$ in~$\sM$, the induced map on derived homs 
        \[ f^*\colon \bR\Hom_\sM(B,X)\to \bR\Hom_\sM(A,X) \]
        is a weak homotopy equivalence.
    \end{itemize}
\end{defn}

\begin{defn}
     The \textbf{$S$-localized model structure $\sM_S$} on $\sM$, if it exists, is such that 
    \begin{rome}
        \item a map $f$ is a trivial fibration in $\sM_S$ if and only if it is a trivial fibration in $\sM$, 
        \item a map $f$ is a weak equivalence in $\sM_S$ if and only if it is an $S$-local equivalence.
    \end{rome}
\end{defn}

We recall the following existence theorem of left Bousfield localizations in the case of left proper, combinatorial model structures, proven in \cite[Theorem 4.7]{Barwick}.

\begin{theorem} \label{thm:locexist}
    Let $\sM$ be a left proper, combinatorial model category and $S$ be a set of morphisms in $\sM$. Then the $S$-localized model structure $\sM_S$ exists. 
    
    Moreover, an object $X\in \sM$ is fibrant in $\sM_S$ if and only if it is $S$-local, and a map $f$ between $S$-local objects is a weak equivalence (resp.~fibration) in $\sM_S$ if and only if it is a weak equivalence (resp.~fibration) in $\sM$. 
\end{theorem}

\subsection{Induced vs localized model structures} \label{sec:indvslocMS}

We now study the interaction between induction and localization. To this end, consider the following setup. Let $\sM$ and $\sN$ be left proper, combinatorial model categories, and suppose that we have a Quillen pair
\begin{tz}
\node[](A) {$\sM$};
\node[right of=A,xshift=.6cm](B) {$\sN$};
\punctuation{B}{.};
\draw[->] ($(A.east)+(0,5pt)$) to node[above,la]{$F$} ($(B.west)+(0,5pt)$);
\draw[->] ($(B.west)-(0,5pt)$) to node[below,la]{$U$} ($(A.east)-(0,5pt)$);
\node[la] at ($(A.east)!0.5!(B.west)$) {$\bot$};
\end{tz}
Given a set $S$ of morphisms in $\sM$, we denote by $\bL F(S)$ the corresponding set of morphisms in $\sN$ induced by the derived functor of $F$. By \cref{thm:locexist}, we obtain localized model structures $\sM_S$ and $\sN_{\bL F(S)}$. The following result appears as \cite[Theorem 3.3.20]{Hirschhorn}.

\begin{prop} \label{prop:QPloc}
    The Quillen pair $F\dashv U$ between $\sM$ and $\sN$ induces a Quillen pair between their localizations
    \begin{tz}
\node[](A) {$\sM_{S}$};
\node[right of=A,xshift=1.1cm](B) {$\sN_{\bL F(S)}$};
\punctuation{B}{.};
\draw[->] ($(A.east)+(0,5pt)$) to node[above,la]{$F$} ($(B.west)+(0,5pt)$);
\draw[->] ($(B.west)-(0,5pt)$) to node[below,la]{$U$} ($(A.east)-(0,5pt)$);
\node[la] at ($(A.east)!0.5!(B.west)$) {$\bot$};
\end{tz}

    Moreover, if $F\dashv U$ is further a Quillen equivalence between $\sM$ and $\sN$, then the induced Quillen pair $F\dashv U$ between $\sM_S$ and $\sN_{\bL F(S)}$ is also a Quillen equivalence. 
\end{prop}

If we now suppose that the model structure on $\sN$ is right-induced from that on $\sM$ along the adjunction $F\dashv U$, we observe that the localized model structure on $\sN$ remains fibrantly-induced from the localized model structure on $\sM$.  

\begin{prop}
If the model structure on $\sN$ is right-induced from that on $\sM$ along the adjunction $F\dashv U$, then the model structure $\sN_{\bL F(S)}$ is fibrantly-induced from $\sM_S$ along the adjunction $F\dashv U$. 
\end{prop}

\begin{proof}
    Recall that the trivial fibrations and weak equivalences (resp.~fibrations) between fibrant objects in the localized model structures $\sM_S$ and $\sN_{\bL F(S)}$ are determined by those of $\sM$ and $\sN$. Since the model structure on $\sN$ is right-induced from that on $\sM$, it is therefore sufficient to show that a fibrant object $Y$ in $\sN$ is $\bL F(S)$-local if and only if $UY$ is $S$-local. But this follows directly from the weak homotopy equivalence between derived homs 
    \[ \bR\Hom_\sN(\bL F(A),Y)\simeq \bR\Hom_\sM(A,UY) \]
    induced by the Quillen pair $F\dashv U$, using that the derived functor $\bR U$ is $U$ itself as it preserves all weak equivalences.
\end{proof}

Finally, we show that, if the adjunction that we started with is further a Quillen equivalence, then the localized model structure on $\sN$ is in fact right-induced from that on~$\sM$.

\begin{prop} \label{prop:Upreswe}
    Suppose that the model structure on $\sN$ is right-induced from that on~$\sM$ along the adjunction $F\dashv U$, and $F\dashv U$ is a Quillen equivalence between $\sM$ and $\sN$. Then the induced right Quillen functor $U\colon \sN_{\bL F(S)}\to \sM_S$ preserves all weak equivalences. 
\end{prop}

\begin{proof}
Let $f\colon A\to B$ be a $\bL F(S)$-local equivalence in $\sN$. Then, by definition, the induced map
\[ f^*\colon \bR \Hom_\sN(B,Y)\to \bR\Hom_\sN(A,Y) \]
is a weak homotopy equivalence, for every $\bL F(S)$-local object $Y\in \sN$. Since $U\colon \sN\to \sM$ is a right Quillen equivalence such that the derived functor $\bR U$ is $U$ itself (as it preserves all weak equivalences), for all objects $A,Y\in \sN$, we have weak homotopy equivalences 
\[ \bR \Hom_\sN(A,Y)\simeq \bR\Hom_\sM(UA,UY). \]
Hence the induced map
\[ (Uf)^*\colon \bR \Hom_\sM(UB,UY)\to \bR\Hom_\sM(UA,UY) \]
is also a weak homotopy equivalence, for every $\bL F(S)$-local object $Y\in \sN$. Since the Quillen equivalence $F\dashv U$ between $\sM$ and $\sN$ induces a Quillen equivalence between the localizations $\sM_S$ and $\sN_{\bL F(S)}$ by \cref{prop:QPloc}, every $S$-local object $X\in\sM$ is weakly equivalent to an object $UY$ with $Y\in \sN$ a $\bL F(S)$-local object. Hence the induced map
\[ (Uf)^*\colon \bR \Hom_\sM(UB,X)\to \bR\Hom_\sM(UA,X) \]
is also a weak homotopy equivalence, for every $S$-local object $X\in \sM$. By definition, this means that $Uf$ is an $S$-local equivalence in $\sM$, as desired. 
\end{proof}

The following is now a direct consequence of \cref{criterionfitori,prop:Upreswe}.

\begin{prop} \label{thm:RIMS}
    Let $\sM$ and $\sN$ be left proper, combinatorial model categories and suppose that the model structure on $\sN$ is right-induced from that on $\sM$ along an adjunction
\begin{tz}
\node[](A) {$\sM$};
\node[right of=A,xshift=.6cm](B) {$\sN$};
\punctuation{B}{.};
\draw[->] ($(A.east)+(0,5pt)$) to node[above,la]{$F$} ($(B.west)+(0,5pt)$);
\draw[->] ($(B.west)-(0,5pt)$) to node[below,la]{$U$} ($(A.east)-(0,5pt)$);
\node[la] at ($(A.east)!0.5!(B.west)$) {$\bot$};
\end{tz}
Given a set $S$ of morphisms in $\sM$, we denote by $\bL F(S)$ the corresponding set of morphisms in $\sN$ induced by the derived functor of $F$. If the adjunction $F\dashv U$ is a Quillen equivalence between $\sM$ and $\sN$, then the localized model structure $\sN_{\bL F(S)}$ is right-induced from $\sM_S$ along the adjunction $F\dashv U$. 
\end{prop}

\subsection{Homotopy colimits} \label{sec:hocolimMS}

Finally, we prove some results concerning homotopy colimit functors in a model category. To this end, we first recall the definition of \emph{projective model structures}. Let $\sM$ be a model category and $\sJ$ be a small category. We consider the category $\sM^\sJ$ of functors $\sJ\to \sM$ and natural transformations between them. 

\begin{defn}
    The \textbf{projective model structure} $\sM^\sJ_\proj$ on $\sM^\sJ$, if it exists, is such that a map $f$ is a weak equivalence (resp.~fibration) in $\sM^\sJ$ if and only if, for every object $j\in \sJ$, the induced map $f_j$ is a weak equivalence (resp.~fibration) in $\sM$. 
\end{defn}

The existence of the projective model structure in the case of a combinatorial model category is well-known, and can be found e.g.~in \cite[Theorem 11.6.1]{Hirschhorn}.

\begin{theorem} \label{thm:existproj}
    Let $\sM$ be a combinatorial model category and $\sJ$ be a small category. Then the projective model structure $\sM^\sJ_\proj$ exists.
\end{theorem}

In the remainder of the section, we assume that $\sM$ is a combinatorial model category, so that the model structure $\sM^\sJ_\proj$ exists by the previous result, and we turn to homotopy colimits. Recall that the constant diagram functor $\cst\colon \sM\to \sM^\sJ$, sending an object $X\in \sM$ to the constant diagram at $X$, admits as a left adjoint the colimit functor $\colim_\sJ\colon \sM^\sJ\to \sM$, and these form a Quillen pair as follows; see e.g.~\cite[Theorem 11.6.8]{Hirschhorn}.

\begin{prop} \label{thm:colimQP}
    The adjunction 
    \begin{tz}
\node[](A) {$\sM^\sJ_{\proj}$};
\node[right of=A,xshift=.9cm](B) {$\sM$};
\draw[->] ($(A.east)+(0,5pt)$) to node[above,la]{$\colim_\sJ$} ($(B.west)+(0,5pt)$);
\draw[->] ($(B.west)-(0,5pt)$) to node[below,la]{$\cst$} ($(A.east)-(0,5pt)$);
\node[la] at ($(A.east)!0.5!(B.west)$) {$\bot$};
\end{tz}
    is a Quillen pair. 
\end{prop}

\begin{defn}
    Let $\cM$ be a model category. A \textbf{cofibrant replacement functor} is a functor $Q \colon \cM \to \cM$ equipped with a natural transformation $\alpha \colon Q \Rightarrow \id_{\cM}$ such that for every object $X$ of $\cM$, the map
    \[
    \alpha_X \colon Q(X) \to X
    \]
    is a weak equivalence, and $Q(X)$ is cofibrant.
\end{defn}

\begin{notation}
    It follows from \cref{thm:colimQP} that the functor $\colim_\sJ$ is left derivable. By choosing a cofibrant replacement functor $Q \colon \sM^\sJ\to \sM^\sJ$
    for the projective model structure, we can define ``the'' homotopy colimit functor as
\[
\hocolim_\sJ\coloneqq \colim_\sJ \circ Q \colon \sM^\sJ\to \sM.
\]
\end{notation}

\begin{defn}
    A functor $ \cL_\sJ \colon \sM^\sJ\to \sM$ is a \textbf{model for the homotopy colimit functor} if it is equivalent to $\hocolim_\sJ$ via a zig-zag of natural transformations which are levelwise weak equivalences in $\sM$, which we simply denote as
\[
\cL_\sJ \simeq \hocolim_\sJ.
\]
\end{defn}

\begin{rem}
    Since $\hocolim_\sJ\colon \sM^\sJ_\proj\to \sM$ preserves all weak equivalences, then so does any model for the homotopy colimit functor.
\end{rem}

\begin{notation}
    Suppose that we are given another small category $\sK$. Given a model of homotopy colimit functor $\cL_\sJ \colon \sM^\sJ\to \sM$, it induces by post-composition a functor 
    \[ (\cL_\sJ)_* \colon \sM^{\sK\times \sJ}\cong (\sM^\sJ)^\sK\to \sM^\sK. \]
\end{notation}

We show that this induced functor is a model of the homotopy colimit functor for the model category $\sM^\sK_{\proj}$. For this, recall that we have a canonical isomorphism of model categories
\[
(\sM^\sK_{\proj})^\sJ_{\proj} \cong \sM^{\sK\times \sJ}_{\proj}.
\]

\begin{prop} \label{hocoliminproj}
    Let $\sM$ be a combinatorial model category, and $\sJ$ and $\sK$ be small categories. Consider a model $\cL_\sJ \colon \sM^\sJ\to \sM$ for the homotopy colimit functor. Then the induced functor
    \[
    (\cL_\sJ)_* \colon (\sM^\sK_\proj)^\sJ\to \sM^\sK_\proj.
    \]
    is a model for the homotopy colimit functor.
\end{prop}

\begin{proof}
Let $Q\colon \sM^{\sK\times \sJ}\to \sM^{\sK\times \sJ}$ and $Q'\colon \sM^\sJ\to \sM^\sJ$ be the chosen cofibrant replacement functors for the model structures $\sM^{\sK\times \sJ}_{\proj}\cong (\sM^\sK_{\proj})^\sJ_{\proj}$ and $\sM^\sJ_\proj$. Then $Q'$ induces by post-composition a functor $Q'_*\colon (\sM^\sK)^\sJ\cong (\sM^\sJ)^\sK\to (\sM^\sJ)^\sK\cong (\sM^\sK)^\sJ$. Then we have a zigzag of natural transformations
\[ Q\xRightarrow{\simeq} QQ'_*\xLeftarrow{\simeq} Q'_* \]
which are levelwise a weak equivalence in $\sM^{\sJ\times \sK}_\proj$. Then, for every functor $F\colon \sK\times \sJ\to \sM$, we have a zigzag of weak equivalences in $\sM^{\sK\times\sJ}_\proj$
\begin{align*} 
\hocolim_\sJ F &=\colim_\sJ Q(F) \xLeftarrow{\simeq} \colim_\sJ QQ'_*(F) \\ &\cong (\colim_\sJ)_* QQ'_*(F)\xRightarrow{\simeq} (\colim_\sJ)_*Q'_*(F)=(\hocolim_\sJ)_* F 
\end{align*}
using that $\colim_\sJ\colon (\sM^\sK)^\sJ\to \sM^\sK$ and $\colim_\sJ\colon \sM^\sJ\to \sM$ preserve weak equivalences between cofibrant objects, and that colimits in $\sM^\sK$ are computed levelwise in $\sM$. As the above zigzag is natural in $F$, we get a zigzag of natural levelwise weak equivalence 
\[ \hocolim_\sJ\simeq (\hocolim_\sJ)_*, \]
as desired. The general statement for an arbitrary model for homotopy colimits follows immediately.
\end{proof}

We conclude these reminders on homotopy colimits with the following result.

\begin{prop} \label{hocoliminloc}
Let $\sM$ be a left proper, combinatorial model structure, $\sJ$ be a small category, and $S$ be a set of morphisms in $\sM$. Consider a model $\cL_\sJ \colon \sM^\sJ\to \sM$ for the homotopy colimit functor. Then the functor 
\[ \cL_\sJ\colon (\sM_S)^\sJ\to \sM_S \]
is also a model for the homotopy colimit functor in the $S$-localized model category $\sM_S$. 
\end{prop}

\begin{proof} 
Since $\sM$ and $\sM_S$ have the same cofibrations, it is straightforward to see that the projective model structure $\sM^\sJ_\proj$ and $(\sM_S)^\sJ_\proj$ have the same cofibrations as well. In particular, if $Q\colon \sM^\sJ\to \sM^\sJ$ is a cofibrant replacement functor for $\sM^\sJ_\proj$, it is also a cofibrant replacement for $(\sM_S)^\sJ_\proj$. The result then follows immediately.
\end{proof}

\section{The horizontal nerve and technical result}

In this section, we review the horizontal nerve functor studied by Fiore--Pronk--Paoli in \cite{FPP}, which maps a double category to a simplicial object in categories. In \cref{sec:hornerve}, we recall the description of the horizontal nerve and its left adjoint. We then prove in \cref{sec:technical} that this nerve preserves certain pushouts---a technical result extending \cite[Theorem 10.7]{FPP} that will be useful later.

\subsection{The horizontal nerve} \label{sec:hornerve}

We denote by $\Cat$ the category of categories and functors. \emph{Double categories} are then defined as internal categories to $\Cat$, and we begin by recalling the notion of internal categories in a category with pullbacks that will be used throughout the paper.

\begin{defn}
    Let $\sM$ be a category with pullbacks. An \textbf{internal category} $\bA$ to $\sM$ is a diagram in $\sM$ of the form
\begin{tz}
    \node[](1) {$\bA_0$};
     \node[right of=1,xshift=.7cm](2) {$\bA_1$}; 
     \node[right of=2,xshift=1.3cm](3) {$\bA_1\times_{\bA_0}\bA_1$}; 

     \draw[->]($(2.west)+(0,6pt)$) to node[above,la]{$s$} ($(1.east)+(0,6pt)$);
     \draw[->]($(2.west)-(0,6pt)$) to node[below,la]{$t$} ($(1.east)-(0,6pt)$);
     \draw[->](1) to node[over,la]{$i$} (2);

     \draw[->](3) to node[above,la]{$c$} (2);
\end{tz}
    consisting of an object $\bA_0$ of \emph{objects}, an object $\bA_1$ of \emph{morphisms}, \emph{source} and \emph{target} morphisms $s$ and $t$, and \emph{identity} and \emph{composition} morphisms $i$ and $c$, such that composition is associative and unital. 

    An \textbf{internal functor} $G\colon \bA\to \bB$ between internal categories to $\sM$ consists of morphisms $G_0\colon \bA_0\to \bB_0$ and $G_1\colon \bA_1\to \bB_1$ in $\sM$ compatible with the structure of internal categories.

    We denote by $\Cat(\sM)$ the category of internal categories to $\sM$ and internal functors.
\end{defn}

We also recall the following construction, which will be used repeatedly.

\begin{defn}
    Let $\sM$ and $\sN$ be categories with pullbacks and $F\colon \sM\to \sN$ be a pullback-preserving functor. Then $F$ induces a functor 
    \[ \Cat(F)\colon \Cat(\sM)\to \Cat(\sN) \]
    between categories of internal categories given by sending an internal category $\bA$ to $\sM$ to the internal category to $\sN$
    \begin{tz}
    \node[](1) {$F(\bA_0)$};
     \node[right of=1,xshift=1.3cm](2) {$F(\bA_1)$}; 
     \node[right of=2,xshift=3.9cm](3) {$F(\bA_1\times_{\bA_0}\bA_1)\cong F(\bA_1)\times_{F(\bA_0)} F(\bA_1)$}; 

     \draw[->]($(2.west)+(0,6pt)$) to node[above,la]{$Fs$} ($(1.east)+(0,6pt)$);
     \draw[->]($(2.west)-(0,6pt)$) to node[below,la]{$Ft$} ($(1.east)-(0,6pt)$);
     \draw[->](1) to node[over,la]{$Fi$} (2);

     \draw[->](3) to node[above,la]{$Fc$} (2);
\end{tz}
    Note that this diagram satisfies the conditions of an internal category as $F$ is functorial. 
\end{defn}

We denote by $\DblCat\coloneqq \Cat(\Cat)$ the category of internal categories and internal functors to $\Cat$, known as \emph{double categories} and \emph{double functors}. Given a double category $\bA$, we refer to $\bA_0$ as the category of \emph{objects} and \emph{vertical morphisms}, and to $\bA_1$ as the category of \emph{horizontal morphisms} and \emph{squares}.  For more details about the category $\DblCat$, we direct the reader to \cite[\textsection3.1.1 and \textsection3.2.1]{Grandis}. We first recall different functors relating the categories $\Cat$ and $\DblCat$.

\begin{defn}
There are canonical inclusions 
\[ \bH\colon \Cat\cong \Cat(\Set)\to \Cat(\Cat)\cong\DblCat \quad \text{and} \quad \bV\colon \Cat\to \Cat(\Cat)\cong \DblCat, \]
given by the functor between categories of internal categories induced by the (pullback-preserving) inclusion $\Set\hookrightarrow \Cat$ and by the constant diagram functor, respectively. 

Explicitly, given a category $\sC$, the horizontal double category $\bH\sC$ has objects and horizontal morphisms given by the objects and morphisms of $\sC$, and only trivial vertical morphisms and squares. The vertical double category $\bV\sC$ admits a similar description by reversing the role of the horizontal and vertical morphisms. 

The \textbf{box product functor} $\boxtimes\colon \Cat\times \Cat\to \DblCat$ is then given by the composite 
\[ \Cat\times \Cat\xrightarrow{\bH\times \bV} \DblCat\times \DblCat \xrightarrow{-\times -}\DblCat \]
sending a pair $(\sC,\sD)$ of categories to the product $\sC\boxtimes \sD\coloneqq \bH\sC\times \bV\sD$. 
\end{defn}

\begin{rem} 
Recall that the inclusion functors $\bH,\bV\colon \Cat\to\DblCat$ admit right adjoint functors $\bfH,\bfV\colon \DblCat\to \Cat$ which send a double category to its underlying horizontal and vertical categories, respectively. 
\end{rem}

We now recall the horizontal nerve from \cite[Definition 5.1]{FPP}. We write $\Simp$ the category of finite posets $[n]\coloneqq \{0<1<\ldots<n\}$ and order-preserving maps between them. Since every poset can be viewed as a category, we often regard $[n]$ as a category without further specification.

\begin{rem}
The category $\DblCat$ is cartesian closed; see \cite[Lemma B2.3.15(ii)]{Elephant}. We denote by $[\bA,\bB]$ the internal hom between two double categories $\bA$ and $\bB$.
\end{rem}

\begin{defn}
The \textbf{horizontal nerve functor} $\Nh\colon \DblCat\to \Cat^{\Simpop}$ sends a double category~$\bA$ to the simplicial object in $\Cat$
\[ \Nh\bA\colon \Simp^\op\to \Cat, \quad [n]\mapsto \bfV[\bH[n],\bA]. \]
Here $\bfV[\bH[n],\bA]$ is the category whose objects are double functors $\bH[n]=[n]\boxtimes[0]\to \bA$ and whose morphisms are double functors $\bH[n]\times \bV[1]=[n]\boxtimes [1]\to \bA$.
\end{defn}

\begin{rem}
Explicitly, we have that $(\Nh\bA)_0$ is the category of objects and vertical morphisms in $\bA$, $(\Nh\bA)_1$ is the category of horizontal morphisms and squares in $\bA$, and 
\[ (\Nh\bA)_n\cong (\Nh\bA)_1\times_{(\Nh\bA)_0}\ldots \times_{(\Nh\bA)_0} (\Nh\bA)_1 \]
is the category of $n$ composable horizontal morphisms and $n$ horizontally composable squares in $\bA$. 
\end{rem}

By \cite[
Theorem 5.6]{FPP}, the horizontal nerve admits a left adjoint. 

\begin{prop}
    The horizontal nerve functor is part of an adjunction
    \begin{tz}
\node[](A) {$\Cat^{\Simpop}$};
\node[right of=A,xshift=1.4cm](B) {$\DblCat$};
\punctuation{B}{.};
\draw[->] ($(A.east)+(0,5pt)$) to node[above,la]{$\ch$} ($(B.west)+(0,5pt)$);
\draw[->] ($(B.west)-(0,5pt)$) to node[below,la]{$\Nh$} ($(A.east)-(0,5pt)$);
\node[la] at ($(A.east)!0.5!(B.west)$) {$\bot$};
\end{tz}
\end{prop}

Finally, we describe the action of the left adjoint on box products, which we also introduce here for $\Cat^{\Simpop}$.

\begin{defn}
    There are canonical inclusions 
    \[ \iota\colon \Set^{\Simpop}\to \Cat^{\Simpop} \quad \text{and}\quad \cst\colon \Cat\to \Cat^{\Simpop} \]
    given by post-composition along the inclusion $\Set\hookrightarrow \Cat$ and by the constant diagram functor, respectively. 

    The \textbf{box product functor} $\boxtimes\colon \Set^{\Simpop}\times \Cat\to \Cat^{\Simpop}$ is then given by the composite
    \[ \Set^{\Simpop}\times \Cat\xrightarrow{\iota\times \cst}\Cat^{\Simpop}\times \Cat^{\Simpop}\xrightarrow{-\times -}\Cat^{\Simpop} \]
    sending a pair $(X,\sC)$ with $X\in \Set^{\Simpop}$ and $\sC\in \Cat$ to the product $X\boxtimes \sC\coloneqq \iota X\times \cst \sC$.
\end{defn}

\begin{rem} \label{rem:usualnerve}
    Recall the usual nerve functor $N\colon \Cat\to \Set^{\Simpop}$ sending a category $\sC$ to the simplicial set 
    \[ N\sC\colon \Simpop\to \Set, \quad [n]\mapsto \Cat([n],\sC). \]
    It is part of an adjunction 
    \begin{tz}
\node[](A) {$\Set^{\Simpop}$};
\node[right of=A,xshift=1.1cm](B) {$\Cat$};
\punctuation{B}{.};
\draw[->] ($(A.east)+(0,5pt)$) to node[above,la]{$c$} ($(B.west)+(0,5pt)$);
\draw[->] ($(B.west)-(0,5pt)$) to node[below,la]{$N$} ($(A.east)-(0,5pt)$);
\node[la] at ($(A.east)!0.5!(B.west)$) {$\bot$};
\end{tz}
\end{rem}

\begin{rem} \label{chonboxprod}
By \cite[Example 6.6]{FPP}, the left adjoint functor $\ch\colon \Cat^{\Simpop}\to \DblCat$ sends the box product $X\boxtimes \sC$ of a simplicial set $X$ and a category $\sC$ to the double category $\bH cX\times\bV\sC= cX\boxtimes \sC$.
\end{rem}

\subsection{Technical result} \label{sec:technical}

We now turn to the technical result. To prepare for this, we first introduce the following terminology.

\begin{defn}
A \textbf{sieve of posets} is an inclusion $P\subseteq Q$ of posets which is
\begin{itemize}[leftmargin=0.6cm]
\item \textbf{full}: for all $p,p'\in P$ such that $p\leq p'$ in $Q$, then $p\leq p'$ in $P$,
\item \textbf{down-closed}: for all $p\in P$ and $q\in Q$ such that $q\leq p$ in $Q$, then $q\in P$.
\end{itemize} 
\end{defn}

We further introduce the following property of sieves of posets, whose name is inspired by its relation to the notion of \emph{solid functors}.

\begin{defn}
   A sieve of posets $P\subseteq Q$ is \textbf{weakly solid} if, for all $p_1,p_2\in P$ and $q\in Q$ such that $p_1,p_2\leq q$, there is an element $p\in P$ such that $p_1,p_2\leq p\leq q$. 
\end{defn}

We begin with the following description of pushouts in $\Cat$ along a coproduct of a sieve of posets. Here, we use the fact that every poset can be viewed as a category and that order-preserving maps between posets correspond to functors between the associated categories. 

\begin{lemma} \label{helpful}
Let $P\subseteq Q$ be a sieve of posets, seen as a functor, and $C$ be a set, seen as a discrete category. Consider the following pushout in $\Cat$. 
\begin{tz}
\node[](1) {$C\times P$}; 
\node[below of=1](2) {$C\times Q$}; 
\node[right of=1,xshift=.7cm](3) {$\sA$}; 
\node[below of=3](4) {$\sP$};
\pushout{4};
\draw[right hook->] (1) to (2);
\draw[->] (1) to node[above,la]{$F$} (3); 
\draw[->] (2) to (4); 
\draw[->] (3) to (4);
\end{tz}
The category $\sP$ admits the following description. Objects in $\sP$ are of two forms: 
    \begin{enumerate}[leftmargin=0.85cm]
        \item an object $a\in \sA$; 
        \item a pair $(c,q)$ of an element $c\in C$ and an element $q\in Q\setminus P$. 
    \end{enumerate}
Morphisms in $\sP$ are of three forms: 
    \begin{enumerate}[leftmargin=0.85cm]
        \item a morphism $a\to b$ in $\sA$; 
        \item a pair $(c,q)\to (c,q')$ of an element $c\in C$ and a relation $q\leq q'$ with $q,q'\in Q\setminus P$; 
        \item a formal composite $a\to F(c,p)\to (c,q)$ of a morphism $a\to F(c,p)$ in $\sA$ with a pair $(c,p)\to (c,q)$ of an element $c\in C$ and a relation $p\leq q$ with $p\in P$ and $q\in Q\setminus P$;
    \end{enumerate}
    with relations coming from the relations in $\bA$ and in $C\times Q$.
\end{lemma}

\begin{proof}
Since the inclusion $P\subseteq Q$ is full, we can apply \cite[Lemma 10.5]{FPP}. We obtain the desired description of the objects and we get that the morphisms in $\sP$ are of two forms:
\begin{rome}
    \item morphisms $(c,q)\to (c,q')$ consisting of an element $c\in C$ and a relation $q\leq q'$ with $q,q'\in Q\setminus P$;
    \item formal composites $x\to a\to b\to y$ where $a\to b$ is in $\sA$ and $x\to a$, $b\to y$ are either in $C\times (\Mor Q\setminus \Mor P)$ or identities.
\end{rome}
We recover the morphisms of type (2) as the morphisms of type (i). We now study the morphisms of type (ii). First, note that, since the inclusion $P\subseteq Q$ is down-closed, if $x\to a$ is in $C\times (\Mor Q\setminus \Mor P)$, then its target has to be in $Q\setminus P$ and so cannot be in $\sA$. Hence $x\to a$ must be an identity. Then the morphisms of type (1) are recovered from the case where $b\to y$ is an identity, and the ones of type (3) from the case where $b\to y$ is in $C\times (\Mor Q\setminus \Mor P)$.
\end{proof}

\begin{rem} \label{rem:descriptiontype3}
   Note that two formal composites of type (3)
   \[ a\to F(c,p_1)\to (c,q) \quad \text{and} \quad a\to F(c,p_2)\to (c,q) \]
   can be equal, and they are equal precisely if we are in one of the following situations: 
        \begin{rome}
            \item there is an element $p\in P$ such that $p_1,p_2\leq p\leq q$ and the composites 
            \[ a\to F(c,p_1) \to F(c,p) \quad \text{and} \quad a\to F(c,p_2)\to F(c,p) \]
            agree in $\sA$,
            \item there is an element $p\in P$ such that $p\leq p_1,p_2$ and a morphism $a\to F(c,p)$ in $\sA$ such that $a\to F(c,p_1)$ and $a\to F(c,p_2)$ factor as
            \[ a\to F(c,p)\to F(c,p_1) \quad \text{and} \quad a\to F(c,p)\to F(c,p_2). \]
        \end{rome}
\end{rem}

Using the above lemma, we give an explicit description of pushouts in $\DblCat$ along a box product of a category with a sieve of posets, assuming the sieve is weakly solid. This condition is necessary for understanding the squares of the resulting pushout double category.

\begin{lemma} \label{technicallemma}
Let $P\subseteq Q$ be a weakly solid sieve of posets, seen as a functor, and $\sC$ be a category. Consider the following pushout in $\DblCat$. 
\begin{tz}
\node[](1) {$\sC\boxtimes P$}; 
\node[below of=1](2) {$\sC\boxtimes Q$}; 
\node[right of=1,xshift=.7cm](3) {$\bA$}; 
\node[below of=3](4) {$\bP$}; 
\pushout{4};
\draw[right hook->] (1) to (2);
\draw[->] (1) to node[above,la]{$F$} (3); 
\draw[->] (2) to (4); 
\draw[->] (3) to (4);
\end{tz}
The double category $\bP$ admits the following description. Objects in $\bP$ are of two forms: 
    \begin{enumerate}[leftmargin=0.85cm]
        \item an object $a\in \bA$; 
        \item a pair $(c,q)$ of an object $c\in \sC$ and an element $q\in Q\setminus P$.
    \end{enumerate}
Horizontal morphisms in $\bP$ are of two forms: 
    \begin{enumerate}[leftmargin=0.85cm]
        \item a horizontal morphism $a\to a'$ in $\bA$; 
        \item a pair $(c,q)\to (c',q)$ of a morphism $c\to c'$ in $\sC$ and an element $q\in Q\setminus P$.
    \end{enumerate}
Vertical morphisms in $\bP$ are of three forms: 
    \begin{enumerate}[leftmargin=0.85cm]
        \item a vertical morphism $a\varrow b$ in $\bA$; 
        \item a pair $(c,q)\varrow (c,q')$ of an object $c\in \sC$ and a relation $q\leq q'$ with $q,q'\in Q\setminus P$;
        \item  a formal composite $a\varrow F(c,p)\varrow (c,q)$ of a vertical morphism $a\varrow F(c,p)$ in $\bA$ with a pair $(c,p)\varrow (c,q)$ of an object $c\in \sC$ and a relation $p\leq q$ with $p\in P$ and $q\in Q\setminus P$.
        \end{enumerate}
        Squares in $\bP$ are of three forms: 
    \begin{enumerate}[leftmargin=0.85cm]
        \item a square $\alpha$ in $\bA$; 
        \item a pair $(c\to c', q\leq q')$ of a morphism $c\to c'$ in $\sC$ and a relation $q\leq q'$ with $q,q'\in Q\setminus P$;
        \begin{tz}
        \node[](1) {$(c,q)$}; 
        \node[below of=1](2) {$(c,q')$}; 
        \node[right of=1,xshift=.9cm](3) {$(c',q)$}; 
        \node[below of=3](4) {$(c',q')$}; 
        
        \draw[->,pro] (1) to (2); 
        \draw[->] (1) to (3);
        \draw[->] (2) to (4); 
        \draw[->,pro] (3) to (4);
        
        \node at ($(1)!0.5!(4)$) {\rotatebox{270}{$\Rightarrow$}};
        \end{tz}
        \item a formal vertical composite of a square in $\bA$ with a pair $(c\to c', p\leq q)$ of a morphism $c\to c'$ in $\sC$ and a relation $p\leq q$ with $p\in P$ and $q\in Q\setminus P$.
        \begin{tz}
        \node[](0) {$a$};
        \node[right of=0,xshift=.9cm](10) {$a'$};
        \node[below of=0](1) {$F(c,p)$}; 
        \node[below of=1](2) {$(c,q)$}; 
        \node[below of=10](3) {$F(c',p)$}; 
        \node[below of=3](4) {$(c',q)$}; 
        
        \draw[->,pro] (0) to (1); 
        \draw[->] (0) to (10);
        \draw[->,pro] (10) to (3);
        \draw[->,pro] (1) to (2); 
        \draw[->] (1) to (3);
        \draw[->] (2) to (4); 
        \draw[->,pro] (3) to (4);
        
        \node at ($(0)!0.5!(3)$) {\rotatebox{270}{$\Rightarrow$}};
        \node at ($(1)!0.5!(4)$) {\rotatebox{270}{$\Rightarrow$}};
        \end{tz}
    \end{enumerate}
    These data satisfy relations coming from the relations in $\bA$ and in $\sC\boxtimes Q$. 
\end{lemma}

\begin{proof}
Since the functor $\bfH\colon \DblCat\to \Cat$ is a left adjoint (see e.g.~\cite[Remark 4.6]{MSV1}) and so preserves pushouts, the underlying horizontal category of $\bP$ is the pushout of underlying horizontal categories, namely it is the pushout in $\Cat$
\begin{tz}
\node[](1) {$\sC\times \Ob P$}; 
\node[below of=1](2) {$\sC\times \Ob Q$}; 
\node[right of=1,xshift=1.1cm](3) {$\bfH\bA$}; 
\node[below of=3](4) {$\bfH\bP$}; 
\pushout{4};
\punctuation{4}{.};
\draw[right hook->] (1) to (2);
\draw[->] (1) to node[above,la]{$\bfH F$} (3); 
\draw[->] (2) to (4); 
\draw[->] (3) to (4);
\end{tz}
Hence, by \cite[Lemma 10.4]{FPP}, we have that $\bfH \bP=\bfH\bA\sqcup (\sC\times (\Ob Q\setminus \Ob P))$, and so the double category $\bP$ has objects and horizontal morphisms as described. 

Similarly, the underlying vertical category of $\bP$ is the pushout of underlying vertical categories, namely it is the pushout in $\Cat$
\begin{tz}
\node[](1) {$\Ob\sC\times P$}; 
\node[below of=1](2) {$\Ob\sC\times Q$}; 
\node[right of=1,xshift=1.1cm](3) {$\bfV\bA$}; 
\node[below of=3](4) {$\bfV\bP$}; 
\pushout{4};
\punctuation{4}{.};
\draw[right hook->] (1) to (2);
\draw[->] (1) to node[above,la]{$\bfV F$} (3); 
\draw[->] (2) to (4); 
\draw[->] (3) to (4);
\end{tz}
By \cref{helpful} applied to $C=\Ob\sC$ and $\sA=\bfV\bA$, we directly get that the vertical morphisms in $\bP$ are of the desired forms. 

We show that the squares in $\bP$ are of the desired forms by induction on the number $n$ of squares in an allowable compatible arrangement of squares in $\bA$ and $\sC\boxtimes Q$; see \cite[Definitions 3.4 and 3.5]{FPP}. If $n=1$, then this square is of type (1), (2), or of type (3) with square in $\bA$ being trivial. If $n>1$, suppose that the full length cut is horizontal. 
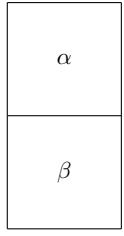
\begin{tz}
\coordinate[](a);
\coordinate[right of=a](b);
\coordinate[below of=a](c); 
\coordinate[below of=b](d); 
\coordinate[below of=c](e); 
\coordinate[below of=d](f);

\draw (a) -- (b);
\draw (a) -- (c);
\draw (b) -- (d);
\draw (c) -- (d);
\draw (e) -- (f);
\draw (c) -- (e);
\draw (d) -- (f);

\node[la] at ($(a)!0.5!(d)$) {$\alpha$}; 
\node[la] at ($(c)!0.5!(f)$) {$\beta$}; 
\end{tz}
By induction, the squares $\alpha$ and $\beta$ are of type (1), (2), or (3). By description of the horizontal morphisms in $\bP$, the full length cut is either completely in $\bA$ or completely in $\Mor\sC\times (Q\setminus P)$. If it is in $\bA$, then $\alpha$ must be of type (1) since this is the only type of square whose bottom boundary is in $\bA$. Then $\beta$ is either of type (1) or (3), and in both cases we can compose to obtain a square of type (1) or (3). If the full length cut is in $\Mor\sC\times (Q\setminus P)$, then the square $\beta$ must be of type (2) since this is the only type of square whose top boundary is in $\Mor\sC\times (Q\setminus P)$. Then $\alpha$ is either of type (2) or (3), and in both cases we can compose to obtain a square of type (2) or (3). 

Now suppose that the full length cut is vertical.
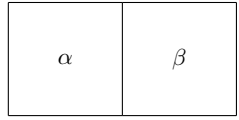
\begin{tz}
\coordinate[](a);
\coordinate[right of=a](b);
\coordinate[below of=a](c); 
\coordinate[below of=b](d); 
\coordinate[right of=b](e); 
\coordinate[below of=e](f);

\draw (a) -- (b);
\draw (a) -- (c);
\draw (b) -- (d);
\draw (c) -- (d);
\draw (e) -- (f);
\draw (b) -- (e);
\draw (d) -- (f);

\node[la] at ($(a)!0.5!(d)$) {$\alpha$}; 
\node[la] at ($(b)!0.5!(f)$) {$\beta$}; 
\end{tz}
By induction, the squares $\alpha$ and $\beta$ are of type (1), (2), or (3). By description of the vertical morphisms, the full length cut is of type (1), (2) or (3). If it is of type (1), i.e., it is completely in $\bA$, then both $\alpha$ and $\beta$ must be of type (1) as these are the only type of squares whose vertical boundaries lie in $\bA$, and we can compose them in $\bA$ to obtain a square of type (1). If the full length cut is of the form (2), i.e., it is in $\Ob\sC\times \Mor(Q\setminus P)$, then both $\alpha$ and $\beta$ must be of type (2) as these are the only type of squares whose vertical boundaries lie in $\Ob\sC\times \Mor(Q\setminus P)$, and we can compose them in $\sC\boxtimes (Q\setminus P)$ to obtain a square of type (2). Finally, if the full length cut is of type (3), then the squares $\alpha$ and $\beta$ must be of the form
\begin{tz}
        \node[](0) {$a_1$};
        \node[right of=0,xshift=.9cm](10) {$a$};
        \node[below of=0](1) {$F(c_1,p_1)$};
        \node[left of=1,xshift=.3cm] {$\alpha=$};
        \node[below of=1](2) {$(c_1,q)$}; 
        \node[below of=10](3) {$F(c,p_1)$}; 
        \node[below of=3](4) {$(c,q)$}; 
        
        \draw[->,pro] (0) to (1); 
        \draw[->] (0) to (10);
        \draw[->,pro] (10) to (3);
        \draw[->,pro] (1) to (2); 
        \draw[->] (1) to (3);
        \draw[->] (2) to (4); 
        \draw[->,pro] (3) to (4);
        
        \node at ($(0)!0.5!(3)$) {\rotatebox{270}{$\Rightarrow$}};
        \node at ($(1)!0.5!(4)$) {\rotatebox{270}{$\Rightarrow$}};

        \node[right of=10,xshift=2cm](0) {$a$};
        \node[right of=0,xshift=.9cm](10) {$a_2$};
        \node[below of=0](1) {$F(c,p_2)$};
        \node[left of=1,xshift=.3cm] {$\beta=$}; 
        \node[below of=1](2) {$(c,q)$}; 
        \node[below of=10](3) {$F(c_2,p_2)$}; 
        \node[below of=3](4) {$(c_2,q)$}; 
        
        \draw[->,pro] (0) to (1); 
        \draw[->] (0) to (10);
        \draw[->,pro] (10) to (3);
        \draw[->,pro] (1) to (2); 
        \draw[->] (1) to (3);
        \draw[->] (2) to (4); 
        \draw[->,pro] (3) to (4);
        
        \node at ($(0)!0.5!(3)$) {\rotatebox{270}{$\Rightarrow$}};
        \node at ($(1)!0.5!(4)$) {\rotatebox{270}{$\Rightarrow$}};
        \end{tz}
        where the top squares are in $\bA$, $c_1\to c$, $c\to c_2$ are morphisms in $\sC$, and $p_1,p_2\leq q$ with $p_1,p_2\in P$ and $q\in Q\setminus P$. Since the right boundary of $\alpha$ must be equal to the left boundary of $\beta$, we have that the composites 
        \[ a\varrow F(c,p_1) \varrow (c,q) \quad \text{and} \quad a\varrow F(c,p_2)\varrow (c,q) \]
        agree in $\bP$. By \cref{rem:descriptiontype3}, we must be in one of the following situations: 
        \begin{rome}
            \item there is an element $p\in P$ such that $p_1,p_2\leq p\leq q$ and the composites 
            \[ a\varrow F(c,p_1) \varrow F(c,p) \quad \text{and} \quad a\varrow F(c,p_2)\varrow F(c,p) \]
            agree in $\bA$,
            \item there is an element $p\in P$ such that $p\leq p_1,p_2$ and a vertical morphism $a\varrow F(c,p)$ in $\bA$ such that $a\varrow F(c,p_1)$ and $a\varrow F(c,p_2)$ factor as
            \[ a\varrow F(c,p)\varrow F(c,p_1) \quad \text{and} \quad a\varrow F(c,p)\varrow F(c,p_2). \]
        \end{rome}
        If we are in case (i), then the horizontal composite of $\alpha$ and $\beta$ is given by 
        \begin{tz}
        \node[](0) {$a_1$};
        \node[right of=0,xshift=.9cm](10) {$a$};
        \node[below of=0](1) {$F(c_1,p_1)$};
        \node[below of=1](2) {$F(c_1,p)$}; 
        \node[below of=2](20) {$(c_1,q)$}; 
        \node[below of=10](3) {$F(c,p_1)$}; 
        \node[below of=3](4) {$F(c,p)$}; 
        \node[below of=4](40) {$(c,q)$}; 
        
        \draw[->,pro] (0) to (1); 
        \draw[->] (0) to (10);
        \draw[->,pro] (10) to (3);
        \draw[->,pro] (1) to (2); 
        \draw[->] (1) to (3);
        \draw[->] (2) to (4); 
        \draw[->] (20) to (40); 
        \draw[->,pro] (3) to (4);
        \draw[->,pro] (2) to (20);
        \draw[->,pro] (4) to (40);
        
        \node at ($(0)!0.5!(3)$) {\rotatebox{270}{$\Rightarrow$}};
        \node at ($(1)!0.5!(4)$) {\rotatebox{270}{$\Rightarrow$}};
        \node at ($(2)!0.5!(40)$) {\rotatebox{270}{$\Rightarrow$}};

        \node[right of=10,xshift=.9cm](0) {$a$};
        \node[below of=0](1) {$F(c,p_2)$}; 
        \node[below of=1](2) {$F(c,p)$}; 
        \node[below of=2](20) {$(c,q)$}; 

        \draw[d](10) to (0); 
        \draw[d](4) to (2);
        \draw[d](40) to (20);
        \draw[->,pro] (2) to (20);

        \node at ($(10)!0.5!(2)$) {$=$};
        \node at ($(4)!0.5!(20)$) {$=$};

        \node[right of=0,xshift=.9cm](10) {$a_2$};
        \node[below of=10](3) {$F(c_2,p_2)$}; 
        \node[below of=3](4) {$F(c_2,p)$}; 
        \node[below of=4](40) {$(c_2,q)$}; 
        
        \draw[->,pro] (0) to (1); 
        \draw[->] (0) to (10);
        \draw[->,pro] (10) to (3);
        \draw[->,pro] (1) to (2); 
        \draw[->] (1) to (3);
        \draw[->] (2) to (4); 
        \draw[->] (20) to (40); 
        \draw[->,pro] (3) to (4);
        \draw[->,pro] (4) to (40);

        \node at ($(0)!0.5!(3)$) {\rotatebox{270}{$\Rightarrow$}};
        \node at ($(1)!0.5!(4)$) {\rotatebox{270}{$\Rightarrow$}};
        \node at ($(2)!0.5!(40)$) {\rotatebox{270}{$\Rightarrow$}};
        \end{tz}
        and so we can compose the top part in $\bA$ and the bottom part in $\sC\boxtimes Q$ to obtain a square of type (3). Finally, case (ii) can be recovered from case (i), since the inclusion $P\subseteq Q$ is weakly solid. Indeed, as $p_1,p_2\leq q$, there is an element $\overline{p}\in P$ such that $p_1,p_2\leq \overline{p}\leq q$ and the composites 
        \[ a\varrow F(c,p)\varrow F(c,p_1)\varrow F(c,\overline{p}) \quad \text{and} \quad a\varrow F(c,p)\varrow F(c,p_2)\varrow F(c,\overline{p}) \]
        agree in $\bA$. This concludes the proof.
\end{proof}

Using this description, we deduce that the horizontal nerve preserves such pushouts.

\begin{prop}\label{pushoutnerve}
Let $P\subseteq Q$ be a weakly solid sieve of posets, seen as a functor, and $\sC$ be a category. Consider the following pushout in $\DblCat$. 
\begin{tz}
\node[](1) {$\sC\boxtimes P$}; 
\node[below of=1](2) {$\sC\boxtimes Q$}; 
\node[right of=1,xshift=.7cm](3) {$\bA$}; 
\node[below of=3](4) {$\bP$}; 
\pushout{4};
\draw[right hook->] (1) to (2);
\draw[->] (1) to node[above,la]{$F$} (3); 
\draw[->] (2) to (4); 
\draw[->] (3) to (4);
\end{tz}
The horizontal nerve $\Nh\colon \DblCat\to \Cat^{\Simp^\op}$ preserves this pushout, meaning that the following square is a pushout in $\Cat^{\Simpop}$.
\begin{tz}
\node[](1) {$\Nh(\sC\boxtimes P)$}; 
\node[below of=1](2) {$\Nh(\sC\boxtimes Q)$}; 
\node[right of=1,xshift=1.4cm](3) {$\Nh\bA$}; 
\node[below of=3](4) {$\Nh\bP$}; 
\pushout{4};
\draw[->] (1) to (2);
\draw[->] (1) to node[above,la]{$\Nh F$} (3); 
\draw[->] (2) to (4); 
\draw[->] (3) to (4);
\end{tz}
\end{prop}

\begin{proof}
The proof proceeds as in \cite[Theorem 10.7]{FPP} and relies on the description of the pushout double category $\bP$ given in \cref{technicallemma}.
\end{proof}

\section{The projective model structure}

In this section, we recall a result by Fiore–Pronk–Paoli in \cite{FPP}, which establishes the existence of the right-induced model structure on $\DblCat$ along the horizontal nerve  $\Nh\colon \DblCat\to \Cat^{\Simpop}$, when $\Cat^{\Simpop}$ is endowed with the projective model structure on simplicial objects in the Thomason model structure on $\Cat$. For this, we first review the main features of the Thomason model structure on $\Cat$ in \cref{sec:Thomason} before recalling the result of \cite{FPP} in \cref{sec:MSdblcat}. Finally, in \cref{sec:leftproper}, we show that the resulting model structure on $\DblCat$ is left proper.

\subsection{Thomason's model structure on \texorpdfstring{$\Cat$}{Cat}} \label{sec:Thomason}

The Thomason model structure on $\Cat$ was introduced by Thomason in \cite{Thomason} and is a model of the homotopy theory of spaces. It is induced from another such model, the Kan--Quillen model structure on simplicial sets, which we begin by recalling. We denote by $\sSet\coloneqq \Set^{\Simpop}$ the category of simplicial sets. 

\begin{notation}
    For $k\geq 0$, we write $\Simp[k]\in \sSet$ for the representable functor at $[k]\in \Simp$. We denote by $\partial \Simp[k]$ its boundary and by $\Lambda^t[k]$ its $t$-horn, for $0\leq t\leq n$. 
\end{notation}

The following result is due to Quillen \cite{Quillen}. 

\begin{theorem}
    There is a model structure on $\sSet$, whose weak equivalences are the weak homotopy equivalences and fibrations are the Kan fibrations. It is called the \emph{Kan--Quillen model structure} and we denote it by $\sSet_\Kan$

    Moreover, the model structure $\sSet_\Kan$ is cofibrantly generated with generating sets of cofibrations and trivial cofibrations given by 
\[ \{ \partial\Simp[k]\to \Simp[k]\mid k\geq 0\} \quad \text{and}\quad \{ \Lambda^t[k]\to \Simp[k]\mid k\geq 1, 0\leq t\leq k\}, \]
respectively, and it is left proper.
\end{theorem}

In order to recall the Thomason model structure, we need the following definition. 

\begin{defn}
For $n\geq 0$, the \textbf{barycentric subdivision} $\Sd\Simp[n]$ is the simplicial set given by the nerve of the poset $\cP\Simp[n]$ of non-degenerate simplices of $\Simp[n]$ ordered by face relations. In particular, the poset $\cP\Simp[n]$ is isomorphic to the poset of non-empty subsets of $\{0,1,\ldots,n\}$ ordered by inclusion. 

These assemble into a cosimplicial object $\Sd\colon \Simp\to \sSet$. By left Kan extending along the Yoneda embedding $\mathsf{y}\colon \Simp\to \sSet$, we get an adjunction 
\begin{tz}
\node[](A) {$\sSet$};
\node[right of=A,xshift=.9cm](B) {$\sSet$};
\punctuation{B}{.};
\draw[->] ($(A.east)+(0,5pt)$) to node[above,la]{$\Sd$} ($(B.west)+(0,5pt)$);
\draw[->] ($(B.west)-(0,5pt)$) to node[below,la]{$\Ex$} ($(A.east)-(0,5pt)$);
\node[la] at ($(A.east)!0.5!(B.west)$) {$\bot$};
\end{tz}
\end{defn}

The Thomason model structure on $\Cat$ is then induced from $\sSet_\Kan$ along the adjunction obtained by iterating the above adjunction twice and further composing with the adjunction $c\dashv N$ recalled in \cref{rem:usualnerve}, where $N\colon \Cat\to \Set^{\Simpop}\cong \sSet$ denotes the usual nerve functor. The following result appears as \cite[Theorem 4.9]{Thomason}.  

\begin{theorem}
The right-induced model structure on $\Cat$ from the Kan--Quillen model structure $\sSet_\Kan$ along the adjunction
\begin{tz}
\node[](A) {$\sSet_\Kan$};
\node[right of=A,xshift=1.2cm](B) {$\Cat$};
\draw[->] ($(A.east)+(0,5pt)$) to node[above,la]{$c\Sd^2$} ($(B.west)+(0,5pt)$);
\draw[->] ($(B.west)-(0,5pt)$) to node[below,la]{$\Ex^2N$} ($(A.east)-(0,5pt)$);
\node[la] at ($(A.east)!0.5!(B.west)$) {$\bot$};
\end{tz}
exists. It is called the \emph{Thomason model structure} and we denote it by $\Cat_\Thom$. 

Moreover, the model structure $\Cat_\Thom$ is cofibrantly generated with generating sets of cofibrations and trivial cofibrations given by 
\[ \{ c\Sd^2\partial\Simp[k]\to c\Sd^2\Simp[k]\mid k\geq 0\} \quad \text{and}\quad \{ c\Sd^2\Lambda^t[k]\to c\Sd^2\Simp[k]\mid k\geq 1, 0\leq t\leq k\} \]
respectively, and it is left proper.
\end{theorem}

In particular, the Thomason model structure on $\Cat$ is Quillen equivalent to the Kan--Quillen model structure on $\sSet$. This result was originally proven in \cite{FL}, and we refer the reader to \cite[Theorem 5.2.12]{Cisinski} for 
a streamlined proof.

\begin{prop} \label{QEThomason}
    The Quillen pair
\begin{tz}
\node[](A) {$\sSet_\Kan$};
\node[right of=A,xshift=1.6cm](B) {$\Cat_\Thom$};
\draw[->] ($(A.east)+(0,5pt)$) to node[above,la]{$c\Sd^2$} ($(B.west)+(0,5pt)$);
\draw[->] ($(B.west)-(0,5pt)$) to node[below,la]{$\Ex^2N$} ($(A.east)-(0,5pt)$);
\node[la] at ($(A.east)!0.5!(B.west)$) {$\bot$};
\end{tz}
    is a Quillen equivalence.
\end{prop}

In fact, the weak equivalences in the Thomason model structure can be described simply as those functors whose nerve is a weak equivalence in $\sSet_\Kan$. The following preliminary result is a direct consequence of \cite[Lemma 3.7]{Kan}.

\begin{prop} \label{natbeta}
    There is a natural transformation of functors $\sSet\to \sSet$
    \[ \beta\colon \id_\sSet\Rightarrow \Ex \]
    which is levelwise a weak equivalence in $\sSet_\Kan$. It induces a natural transformation of functors $\Cat\to \sSet$
    \[ \beta^2N\colon N\Rightarrow \Ex^2 N \]
    which is levelwise a weak equivalence in $\sSet_\Kan$. 
\end{prop}

From this result and the definition of the weak equivalences in $\Cat_\Thom$, we deduce the desired characterization, originally proven in \cite[Proposition 2.4]{Thomason}.

\begin{prop} \label{cor:weThomason}
    A functor $F$ is a weak equivalence in $\Cat_\Thom$ if and only if its nerve~$NF$ is a weak equivalence in $\sSet_\Kan$.
\end{prop}

\subsection{Fiore--Pronk--Paoli's model structure on \texorpdfstring{$\DblCat$}{DblCat}} \label{sec:MSdblcat}

We now recall the desired result from \cite{FPP}. The resulting model structure on $\DblCat$ is induced from the projective model structure on $\Cat^{\Simpop}$, where $\Cat$ is endowed with the Thomason model structure.

\begin{prop} \label{Thomproj}
    The projective model structure on $(\Cat_\Thom)^{\Simpop}$ exists, and we denote it simply by $\projThom$.

    Moreover, the model structure $\projThom$ is cofibrantly generated with generating set of cofibrations 
\[ \{ \Simp[n]\boxtimes (c\Sd^2\partial\Simp[k]\to c\Sd^2\Simp[k])\mid n\geq 0, k\geq 0\} \]
and generating set of trivial cofibrations
\[ \{ \Simp[n]\boxtimes (c\Sd^2\Lambda^t[k]\to c\Sd^2\Simp[k])\mid n\geq 0, k\geq 1, 0\leq t\leq k\}, \]
and it is left proper.
\end{prop}

By right-inducing along the horizontal nerve, we obtain the desired model structure on $\DblCat$, whose existence is proven in \cite[Theorem 7.13]{FPP}.

\begin{theorem} \label{projMS}
The right-induced model structure on $\DblCat$ from the projective model structure $\projThom$ along the adjunction
\begin{tz}
\node[](A) {$\projThom$};
\node[right of=A,xshift=1.4cm](B) {$\DblCat$};
\draw[->] ($(A.east)+(0,5pt)$) to node[above,la]{$\ch$} ($(B.west)+(0,5pt)$);
\draw[->] ($(B.west)-(0,5pt)$) to node[below,la]{$\Nh$} ($(A.east)-(0,5pt)$);
\node[la] at ($(A.east)!0.5!(B.west)$) {$\bot$};
\end{tz}
exists. We denote it by $\DblCat_\proj$.

Moreover, the model structure $\DblCat_\proj$ is cofibrantly generated with generating set of cofibrations 
\[ \{ [n]\boxtimes (c\Sd^2\partial\Simp[k]\to c\Sd^2\Simp[k])\mid n\geq 0, k\geq 0\} \]
and generating set of trivial cofibrations
\[ \{ [n]\boxtimes (c\Sd^2\Lambda^t[k]\to c\Sd^2\Simp[k])\mid n\geq 0, k\geq 1, 0\leq t\leq k\}. \]
\end{theorem}

Finally, we characterize the weak equivalences in the above model structure through their double nerve. For this, we first relate $\Cat^{\Simpop}_\proj$ with the analogous model structure coming from $\sSet_\Kan$.

\begin{prop} \label{projKan}
    The projective model structure on $(\sSet_\Kan)^{\Simpop}$ exists, and we denote it simply by $\sSet^{\Simpop}_\proj$. 
    
    Moreover, the model structure $\sSet^{\Simpop}_\proj$ is cofibrantly generated, and left proper. 
\end{prop}

The Quillen equivalence $c\Sd^2\dashv \Ex^2N$ from \cref{QEThomason} induces by post-composition the following Quillen equivalence.

\begin{prop}\label{QEprojThomason}
    The adjunction
    \begin{tz}
\node[](A) {$\projKan$};
\node[right of=A,xshift=1.5cm](B) {$\projThom$};
\draw[->] ($(A.east)+(0,5pt)$) to node[above,la]{$(c\Sd^2)_*$} ($(B.west)+(0,5pt)$);
\draw[->] ($(B.west)-(0,5pt)$) to node[below,la]{$(\Ex^2N)_*$} ($(A.east)-(0,5pt)$);
\node[la] at ($(A.east)!0.5!(B.west)$) {$\bot$};
\end{tz}
 is a Quillen equivalence, where $\projThom$ is right-induced from $\projKan$.  
\end{prop} 

\begin{proof}
    The fact that it is a Quillen equivalence is a direct consequence of \cref{QEThomason} and \cite[Theorem 11.6.5]{Hirschhorn}. The fact that $\projThom$ is right-induced from $\projKan$ follows from the definition of the projective model structures and the definition the model structure $\Cat_\Thom$ as right-induced from $\sSet_\Kan$ along the adjunction $c\Sd^2\dashv \Ex^2 N$.
\end{proof}

We now recall the double nerve functor and prove the desired characterization. 

\begin{rem} \label{rem:doublenerve}
    Recall the double nerve functor $\bN\colon \DblCat\to \Set^{\Simpop\times \Simpop}\cong \sSet^{\Simpop}$ sending a double category $\bA$ to the simplicial set 
    \[ \bN\bA\colon \Simpop\times \Simpop\to \Set, \quad ([n],[k])\mapsto \DblCat([n]\boxtimes [k],\bA). \]
    It is part of an adjunction 
    \begin{tz}
\node[](A) {$\sSet^{\Simpop}$};
\node[right of=A,xshift=1.4cm](B) {$\DblCat$};
\punctuation{B}{.};
\draw[->] ($(A.east)+(0,5pt)$) to node[above,la]{$\bC$} ($(B.west)+(0,5pt)$);
\draw[->] ($(B.west)-(0,5pt)$) to node[below,la]{$\bN$} ($(A.east)-(0,5pt)$);
\node[la] at ($(A.east)!0.5!(B.west)$) {$\bot$};
\end{tz}
    A straightforward computation further shows that the double nerve functor coincides with the composite 
    \[ \DblCat\xrightarrow{N_h}\Cat^{\Simpop} \xrightarrow{N_*} \sSet^{\Simpop},  \]
    where $N_*$ is induced by post-composition along the usual nerve. 
\end{rem}

\begin{prop} \label{natbetadouble}
    The natural transformation of functors $\Cat^{\Simpop}\to \sSet^{\Simpop}$
    \[ (\beta^2N)_*\colon N_*\Rightarrow(\Ex^2N)_* \]
    induced by the natural transformation from \cref{natbeta} is levelwise a weak equivalence in $\sSet^{\Simpop}_\proj$. It induces a natural transformation of functors $\DblCat\to \sSet^{\Simpop}$
    \[ (\beta^2N)_*\Nh\colon \bN\cong N_*\Nh \Rightarrow (\Ex^2N)_*\Nh \]
    which is levelwise a weak equivalence in $\sSet^{\Simpop}_\proj$. 
\end{prop}

\begin{proof}
    This follows directly from \cref{natalpha}, using that, by definition, a map is a weak equivalence in $\sSet^{\Simpop}_\proj$ if and only if it is levelwise a weak equivalence in $\sSet_\Kan$.
\end{proof}

\begin{prop} \label{prop:charwe}
    A double functor $F$ is a weak equivalence in $\DblCat_\proj$ if and only if its double nerve $\bN F$ in $\sSet^{\Simpop}$ is a weak equivalence in $\sSet^{\Simpop}_\proj$. 
\end{prop}

\begin{proof}
    This follows directly from \cref{natbetadouble} and the fact that the weak equivalences in $\DblCat_\proj$ are induced from those of $\sSet^{\Simpop}_\proj$ along the right adjoint functor $(\Ex^2N)_* N^h$ by \cref{projMS,QEprojThomason}. 
\end{proof}

\subsection{Left properness} \label{sec:leftproper}

We now aim to show that the model structure $\DblCat_\proj$ is left proper. To achieve this, we make use of \cref{pushoutnerve} and start by proving the following lemma.

\begin{lemma}\label{subdivincl}
For $k\geq 0$, the inclusion $c\Sd^2\partial\Simp[k]\to c\Sd^2\Simp[k]$ is a weakly solid sieve of posets.
\end{lemma}

\begin{proof}
      The fact that the inclusion is a sieve follows from \cite[Proposition 4.2]{Thomason}.

      We prove that the inclusion is weakly solid. For this, recall that $c\Sd^2 \Simp[k]$ is the poset 
      \[ \{ (I_1,\ldots,I_r) \mid r\geq 1,\, I_1\subsetneq \ldots\subsetneq I_r\subseteq \{0,\ldots,k\} \}\]
      with order given by $(I_1,\ldots,I_r)\leq (J_1,\ldots,J_s)$ if and only if $\{I_1,\ldots,I_r\}\subseteq \{J_1,\ldots,J_s\}$. Then we see that $c\Sd^2\partial\Simp[k]$ is the subposet of $c\Sd^2 \Simp[k]$ containing the elements $(I_1,\ldots,I_r)$ such that $I_r\neq \{0,\ldots,k\}$. Now let $(I_1,\ldots,I_r)$ and $(J_1,\ldots,J_s)$ be elements in $c\Sd^2 \partial\Simp[k]$ and  $(K_1,\ldots,K_t)$ be an element in $c\Sd^2 \Simp[k]$ such that $(I_1,\ldots,I_r), (J_1,\ldots,J_s)\leq (K_1,\ldots,K_t)$. If $(K_1,\ldots,K_t)$ is in $c\Sd^2\partial\Simp[k]$, there is nothing to prove. Otherwise, we have $K_t=\{0,\ldots,k\}$ and so $(K_1,\ldots,K_{t-1})$ is an element in $c\Sd^2 \partial\Simp[k]$ is such that 
      \[ (I_1,\ldots,I_r), (J_1,\ldots,J_s)\leq (K_1,\ldots,K_{t-1})\leq (K_1,\ldots,K_t), \] as desired.
\end{proof}

\begin{rem}
    A similar proof shows that, for $k\geq 1$ and $0\leq t\leq k$, the inclusion functor $c\Sd^2\Lambda^t[k]\to c\Sd^2\Simp[k]$ is also a weakly solid sieve of posets. Hence \cref{technicallemma} gives another proof of the pushout description provided in \cite[Theorem 10.6]{FPP} in the case of the functor $c\Sd^2\Lambda^t[k]\to c\Sd^2\Simp[k]$.
\end{rem}

\begin{theorem} \label{thm:dblcatleftproper}
The model structure $\DblCat_\proj$ is left proper. 
\end{theorem}

\begin{proof}
First note that the weak equivalences in $\DblCat_\proj$ are closed under filtered colimits. Indeed, by \cref{prop:charwe} they are created by the double nerve 
\[ \bN\colon \DblCat_\proj\to \sSet^{\Simpop}_\proj \]
which preserves filtered colimits, and the weak equivalences in $\sSet^{\Simpop}_\proj$ are closed under filtered colimits.

Hence, to show that the model structure $\DblCat_\proj$ is left proper, it is enough to show the following: for every generating cofibration $[n]\boxtimes (c\Sd^2\partial\Simp[k]\to c\Sd^2\Simp[k])$ in $\DblCat_\proj$, with $n\geq 0$ and $k\geq 0$, and every commutative diagram in $\DblCat$
\begin{tz}
\node[](1) {$[n]\boxtimes c\Sd^2\partial\Simp[k]$};
\node[below of=1](2) {$[n]\boxtimes c\Sd^2\Simp[k]$}; 
\node[right of=1,xshift=1.4cm](3) {$\bA$}; 
\node[below of=3](4) {$\bB$}; 
\node[right of=3,xshift=.2cm](5) {$\bC$}; 
\node[below of=5](6) {$\bD$}; 
\pushout{4};
\pushout{6};
\punctuation{6}{,};

\draw[right hook->] (1) to (2); 
\draw[->] (1) to (3); 
\draw[->] (2) to (4); 
\draw[right hook->] (3) to (4); 
\draw[->] (3) to node[above,la]{$\sim$} (5); 
\draw[->] (4) to (6); 
\draw[right hook->] (5) to (6); 
\end{tz}
where both squares are pushouts and the double functor $\bA\to \bC$ is a weak equivalence in $\DblCat_\proj$, then the double functor $\bB\to \bD$ is also a weak equivalence in $\DblCat_\proj$. 

By applying the horizontal nerve $\Nh$ to the above diagram, we get a commutative diagram in $\Cat^{\Simpop}$
\begin{tz}
\node[](1) {$\Simp[n]\boxtimes c\Sd^2\partial\Simp[k]$};
\node[below of=1](2) {$\Simp[n]\boxtimes c\Sd^2\Simp[k]$}; 
\node[right of=1,xshift=1.8cm](3) {$\Nh\bA$}; 
\node[below of=3](4) {$\Nh\bB$}; 
\node[right of=3,xshift=.8cm](5) {$\Nh\bC$}; 
\node[below of=5](6) {$\Nh\bD$}; 
\pushout{4};
\punctuation{6}{,};

\draw[right hook->] (1) to (2); 
\draw[->] (1) to (3); 
\draw[->] (2) to (4); 
\draw[right hook->] (3) to (4); 
\draw[->] (3) to node[above,la]{$\sim$} (5); 
\draw[->] (4) to (6); 
\draw[right hook->] (5) to (6); 
\end{tz}
where the left-hand and outer squares are pushouts by \cref{pushoutnerve,subdivincl}. Hence, the right-hand square is also a pushout, in which the map $\Nh\bA\to \Nh\bB$ is a cofibration (as a pushout of a generating cofibration) and the map $\Nh\bA\to \Nh\bC$ is a weak equivalence in $\projThom$ (by definition of the weak equivalence $\bA\to \bC$ in $\DblCat_\proj$). By left properness of $\projThom$, we get that the map $\Nh\bB\to \Nh\bD$ is a weak equivalence in $\projThom$ showing that $\bB\to \bD$ is a weak equivalence in $\DblCat_\proj$, as desired. 
\end{proof}

\section{The Quillen equivalence}

In this section, we aim to show that the model structure $\DblCat_\proj$ from \cref{projMS} is in fact Quillen equivalent to the model structure $\Cat^{\Simpop}_\proj$ from \cref{Thomproj}. To establish this, we use an analogous result by Horel in \cite{Horel}, who proves that there is a model structure on the category of internal categories to $\sSet$ which is induced from and Quillen equivalent to the projective model structure on $\sSet^{\Simpop}$ for simplicial objects in the Kan--Quillen model structure on $\sSet$. We first recall Horel’s result in \cref{sec:HorelMS}. Then, in \cref{sec:QP}, we relate Horel’s setting to ours via a commutative square of Quillen pairs, and in \cref{sec:QE}, we show that all of these Quillen pairs are Quillen equivalences, thereby achieving our goal.

\subsection{Horel's model structure on \texorpdfstring{$\Cat(\sSet)$}{Cat(sSet)}} \label{sec:HorelMS}

We denote by $\Cat(\sSet)$ the category of internal categories to simplicial sets and internal functors. Similarly to $\DblCat$ and $\Cat^{\Simpop}$, we introduce a box product for $\Cat(\sSet)$.

\begin{defn}
    There are canonical inclusions 
    \[ \iota\colon \Cat\cong \Cat(\Set)\to \Cat(\sSet) \quad \text{and} \quad \cst\colon \sSet\to \Cat(\sSet) \]
    given by the functor between categories of internal categories induced by the (pullback-preserving) inclusion $\Set\hookrightarrow \sSet$ and by the constant diagram functor, respectively. 

    The \textbf{box product functor} $\boxtimes\colon \Cat\times \sSet\to \Cat(\sSet)$ is then given by the composite
    \[ \Cat\times \sSet\xrightarrow{\iota\times \cst}\Cat(\sSet)\times \Cat(\sSet)\xrightarrow{-\times -} \Cat(\sSet)\]
    sending a pair $(\sC,X)$ with $\sC\in \Cat$ and $X\in \sSet$ to the product $\sC\boxtimes X\coloneqq \iota\sC\times \cst X$.
\end{defn}

\begin{rem}
    The full embedding $\cst\colon \sSet\to \Cat(\sSet)$ admits as a right adjoint the functor $(-)_0\colon \Cat(\sSet)\to \sSet$ which sends an internal category to $\sSet$ to its underlying simplicial set of objects. 
\end{rem}

Similarly to the model structure $\DblCat_\proj$, Horel's model structure on $\Cat(\sSet)$ is also defined as a right-induced model structure along a nerve, introduced in \cite[\textsection 3.6]{Horel}, which we now recall. 

\begin{rem}
The category $\Cat(\sSet)$ is cartesian closed; see \cite[Lemma B2.3.15(ii)]{Elephant} or \cite[Proposition 3.5]{Horel}. We denote by $[\bX,\bY]$ the internal hom between two internal categories $\bX$ and $\bY$ to $\sSet$.
\end{rem}

\begin{defn}
The \textbf{simplicial nerve functor} $N^\Simp\colon \Cat(\sSet)\to \sSet^{\Simpop}$ sends an internal category $\bX$ to $\sSet$ to the simplicial object in $\sSet$
\[ N^\Simp\bX\colon \Simpop\to \sSet, \quad [n]\mapsto [\iota[n],\bX]_0. \]
Here $[\iota[n],\bX]_0$ is the simplicial set whose value at $k\geq 0$ is given by the set of internal functor $[n]\boxtimes \Simp[k]\to \bX$. 
\end{defn}

By \cite[\textsection 3.6]{Horel}, the simplicial nerve admits a left adjoint. 

\begin{prop}
    The simplicial nerve functor is part of an adjunction
    \begin{tz}
\node[](A) {$\sSet^{\Simpop}$};
\node[right of=A,xshift=1.6cm](B) {$\Cat(\sSet)$};
\punctuation{B}{.};
\draw[->] ($(A.east)+(0,5pt)$) to node[above,la]{$c^\Simp$} ($(B.west)+(0,5pt)$);
\draw[->] ($(B.west)-(0,5pt)$) to node[below,la]{$N^\Simp$} ($(A.east)-(0,5pt)$);
\node[la] at ($(A.east)!0.5!(B.west)$) {$\bot$};
\end{tz}
\end{prop}

 Horel's model structure on $\Cat(\sSet)$ is then right-induced from the projective model structure $\projKan$ along the simplicial nerve. The existence of such a model structure is proven in \cite[Theorem 5.2]{Horel}, and its left properness is established in \cite[Proposition 5.3]{Horel}.

\begin{theorem} \label{HorelMS}
The right-induced model structure on $\Cat(\sSet)$ from the projective model structure $\projKan$ along the adjunction
\begin{tz}
\node[](A) {$\projKan$};
\node[right of=A,xshift=1.7cm](B) {$\Cat(\sSet)$};
\draw[->] ($(A.east)+(0,5pt)$) to node[above,la]{$c^\Simp$} ($(B.west)+(0,5pt)$);
\draw[->] ($(B.west)-(0,5pt)$) to node[below,la]{$N^\Simp$} ($(A.east)-(0,5pt)$);
\node[la] at ($(A.east)!0.5!(B.west)$) {$\bot$};
\end{tz}
exists. We denote it by $\Cat(\sSet)_\proj$.

Moreover, the model structure $\Cat(\sSet)_\proj$ is cofibrantly generated, and left proper.
\end{theorem}

The model structure on $\Cat(\sSet)$ is in fact Quillen equivalent to the projective model structure $\projKan$, as shown in \cite[Proposition 5.5]{Horel}. 

\begin{prop} \label{QEHorel}
    The Quillen pair
\begin{tz}
\node[](A) {$\projKan$};
\node[right of=A,xshift=1.9cm](B) {$\Cat(\sSet)_\proj$};
\draw[->] ($(A.east)+(0,5pt)$) to node[above,la]{$c^\Simp$} ($(B.west)+(0,5pt)$);
\draw[->] ($(B.west)-(0,5pt)$) to node[below,la]{$N^\Simp$} ($(A.east)-(0,5pt)$);
\node[la] at ($(A.east)!0.5!(B.west)$) {$\bot$};
\end{tz}
    is a Quillen equivalence.
\end{prop}

\subsection{Four Quillen pairs} \label{sec:QP}

Since the right adjoint $\Ex^2N$ of the below left adjunction preserves pullbacks, it induces a functor $\Cat(\Ex^2N)\colon\DblCat\cong \Cat(\Cat)\to \Cat(\sSet)$ between categories of internal categories,
which is part of an adjunction as below right.
\begin{tz}
\node[](A) {$\sSet$};
\node[right of=A,xshift=.9cm](B) {$\Cat$};
\draw[->] ($(A.east)+(0,5pt)$) to node[above,la]{$c\Sd^2$} ($(B.west)+(0,5pt)$);
\draw[->] ($(B.west)-(0,5pt)$) to node[below,la]{$\Ex^2 N$} ($(A.east)-(0,5pt)$);
\node[la] at ($(A.east)!0.5!(B.west)$) {$\bot$};

\node[right of=B,xshift=2cm](A) {$\Cat(\sSet)$};
\node[right of=A,xshift=1.7cm](B) {$\DblCat$};
\draw[->] ($(A.east)+(0,5pt)$) to node[above,la]{$\Cat(c\Sd^2)$} ($(B.west)+(0,5pt)$);
\draw[->] ($(B.west)-(0,5pt)$) to node[below,la]{$\Cat(\Ex^2 N)$} ($(A.east)-(0,5pt)$);
\node[la] at ($(A.east)!0.5!(B.west)$) {$\bot$};
\end{tz}
Despite the notation, the left adjoint $\Cat(c\Sd^2)$ is \emph{not} given by applying the functor $c\Sd^2$ levelwise, as the latter does not preserve pullbacks. However, its existence is guaranteed by the application of the Adjoint Functor Theorem, and we can describe its action on box products as follows. 

\begin{rem} \label{cat(csd)onboxprod}
    The left adjoint $\Cat(c\Sd^2)\colon \Cat(\sSet)\to \DblCat$ sends the box product $\sC\boxtimes X$ of a category $\sC$ and a simplicial set $X$ to the double category $\sC\boxtimes c\Sd^2 X$. 
\end{rem}

We now show that we have a commutative square of adjunctions relating $\sSet^{\Simpop}$, $\Cat^{\Simpop}$, $\Cat(\sSet)$, and $\DblCat$. For this, we first introduce a box product functor for the category $\sSet^{\Simpop}$ as well. 

\begin{defn}
    There are canonical inclusions 
    \[ \iota\colon \Set^{\Simpop}\to \sSet^{\Simpop}\quad \text{and} \quad \cst\colon \sSet\to \sSet^{\Simpop} \]
    given by post-composition along the inclusion $\Set\hookrightarrow \sSet$ and by the constant diagram functor, respectively. 
    
    The \textbf{box product} $\boxtimes\colon \Set^{\Simpop}\times \sSet\to \sSet^{\Simpop}$ is then given by the composite 
    \[ \Set^{\Simpop}\times \sSet\hookrightarrow \sSet^{\Simpop}\times \sSet^{\Simpop} \xrightarrow{\times} \sSet^{\Simpop} \]
    sending a pair $(X,Y)$ with $X\in \Set^{\Simpop}$ and $Y\in \sSet$ to the product $X\boxtimes Y\coloneqq \iota X\times \cst Y$. 
\end{defn}

\begin{prop} \label{commutativesquare}
We have a commutative square of adjunctions 
\begin{tz}
    \node[](1) {$\sSet^{\Simpop}$}; 
    \node[below of=1,yshift=-.5cm](2) {$\Cat(\sSet)$}; 
    \node[right of=1,xshift=1.7cm](3) {$\Cat^{\Simpop}$}; 
    \node[below of=3,yshift=-.5cm](4) {$\DblCat$};
\punctuation{4}{.};

    \draw[->] ($(1.east)+(0,5pt)$) to node[above,la]{$(c\Sd^2)_*$} ($(3.west)+(0,5pt)$);
\draw[->] ($(3.west)-(0,5pt)$) to node[below,la]{$(\Ex^2 N)_*$} ($(1.east)-(0,5pt)$);
\node[la] at ($(1.east)!0.5!(3.west)$) {$\bot$};
\draw[->] ($(2.east)+(0,5pt)$) to node[above,la]{$\Cat(c\Sd^2)$} ($(4.west)+(0,5pt)$);
\draw[->] ($(4.west)-(0,5pt)$) to node[below,la]{$\Cat(\Ex^2 N)$} ($(2.east)-(0,5pt)$);
\node[la] at ($(2.east)!0.5!(4.west)$) {$\bot$};

\draw[->] ($(3.south)-(5pt,0)$) to node[left,la]{$\ch$} ($(4.north)-(5pt,0)$); 
\draw[->] ($(4.north)+(5pt,0)$) to node[right,la]{$\Nh$} ($(3.south)+(5pt,0)$); 
\node[la] at ($(3.south)!0.5!(4.north)$) {\rotatebox{90}{$\bot$}};
\draw[->] ($(1.south)-(5pt,0)$) to node[left,la]{$c^\Simp$} ($(2.north)-(5pt,0)$); 
\draw[->] ($(2.north)+(5pt,0)$) to node[right,la]{$N^\Simp$} ($(1.south)+(5pt,0)$); 
\node[la] at ($(1.south)!0.5!(2.north)$) {\rotatebox{90}{$\bot$}};
\end{tz}
\end{prop}

\begin{proof}
    We show that the square of left adjoints commutes. Since $\sSet^{\Simpop}$ is a presheaf category, it is enough to show that the composite of left adjoints agree on representables. Let $\Simp[n]\boxtimes \Simp[k]$ be a representable in $\sSet^{\Simpop}$, for $n,k\geq 0$. Then, using \cref{chonboxprod}, we have isomorphisms in $\DblCat$ 
    \[ c^h((c\Sd^2)_*(\Simp[n]\boxtimes \Simp[k]))\cong c^h(\Simp[n]\boxtimes c\Sd^2 \Simp[k]) \cong [n]\boxtimes c\Sd^2 \Simp[k] \]
    and, using \cref{cat(csd)onboxprod}, we have isomorphisms in $\DblCat$ 
    \[ \Cat(c\Sd^2)(c^\Simp(\Simp[n]\boxtimes \Simp[k])) \cong \Cat(c\Sd^2)([n]\boxtimes \Simp[k]) \cong [n]\boxtimes c\Sd^2\Simp[k].\]
    As the above isomorphisms are natural in $n,k$, we get the desired natural isomorphism $c^h(c\Sd^2)_*\cong \Cat(c\Sd^2)c^\Simp$.
\end{proof}

 As a consequence, we obtain the following result. 

 \begin{theorem} \label{thm:Quillenpairs}
    We have a commutative square of Quillen pairs
    \begin{tz}
    \node[](1) {$\sSet^{\Simpop}_\proj$}; 
    \node[below of=1,yshift=-.5cm](2) {$\Cat(\sSet)_\proj$}; 
    \node[right of=1,xshift=2.2cm](3) {$\Cat^{\Simpop}_\proj$}; 
    \node[below of=3,yshift=-.5cm](4) {$\DblCat_\proj$};
\punctuation{4}{,};

    \draw[->] ($(1.east)+(0,5pt)$) to node[above,la]{$(c\Sd^2)_*$} ($(3.west)+(0,5pt)$);
\draw[->] ($(3.west)-(0,5pt)$) to node[below,la]{$(\Ex^2 N)_*$} ($(1.east)-(0,5pt)$);
\node[la] at ($(1.east)!0.5!(3.west)$) {$\bot$};
\draw[->] ($(2.east)+(0,5pt)$) to node[above,la]{$\Cat(c\Sd^2)$} ($(4.west)+(0,5pt)$);
\draw[->] ($(4.west)-(0,5pt)$) to node[below,la]{$\Cat(\Ex^2 N)$} ($(2.east)-(0,5pt)$);
\node[la] at ($(2.east)!0.5!(4.west)$) {$\bot$};

\draw[->] ($(3.south)-(5pt,0)$) to node[left,la]{$\ch$} ($(4.north)-(5pt,0)$); 
\draw[->] ($(4.north)+(5pt,0)$) to node[right,la]{$\Nh$} ($(3.south)+(5pt,0)$); 
\node[la] at ($(3.south)!0.5!(4.north)$) {\rotatebox{90}{$\bot$}};
\draw[->] ($(1.south)-(5pt,0)$) to node[left,la]{$c^\Simp$} ($(2.north)-(5pt,0)$); 
\draw[->] ($(2.north)+(5pt,0)$) to node[right,la]{$N^\Simp$} ($(1.south)+(5pt,0)$); 
\node[la] at ($(1.south)!0.5!(2.north)$) {\rotatebox{90}{$\bot$}};
\end{tz}
    where all model structures are right-induced from $\sSet^{\Simpop}_\proj$. 
\end{theorem}

\begin{proof}
    By definition, the model structures $\Cat(\sSet)_\proj$ and $\DblCat_\proj$ are right-induced along the vertical adjunctions from $\projKan$ and $\projThom$, respectively, and the model structure $\projThom$ is right-induced along the top adjunction from $\projKan$ by \cref{QEprojThomason}. As a consequence, the model structure $\Cat(\sSet)_\proj$ is also right-induced from $\DblCat_\proj$ along the bottom adjunction, since the square commutes by \cref{commutativesquare}.
\end{proof}

\subsection{Four Quillen equivalences} \label{sec:QE}

We now aim to show that the Quillen pairs in \cref{thm:Quillenpairs} are all Quillen equivalences. Since the left and top adjunctions in the diagram have already been established to be Quillen equivalences by \cref{QEHorel,QEprojThomason}, it suffices to prove that either the right or bottom adjunction is a Quillen equivalence. We prove the following.

\begin{theorem}\label{thm:quilleneq}
    The Quillen pair 
\begin{tz}
\node[](A) {$\Cat(\sSet)_\proj$};
\node[right of=A,xshift=2.2cm](B) {$\DblCat_\proj$};
\draw[->] ($(A.east)+(0,5pt)$) to node[above,la]{$\Cat(c\Sd^2)$} ($(B.west)+(0,5pt)$);
\draw[->] ($(B.west)-(0,5pt)$) to node[below,la]{$\Cat(\Ex^2 N)$} ($(A.east)-(0,5pt)$);
\node[la] at ($(A.east)!0.5!(B.west)$) {$\bot$};
\end{tz}
    is a Quillen equivalence.
\end{theorem}

To prove this result, we introduce the following notations. 

\begin{notation}
    Let $\sM$ be a relative category $\sM$, i.e., a category $\sM$ equipped with a wide subcategory, whose morphisms are referred to as \emph{weak equivalences}. We denote by $\sM_{\infty}$ its underlying $\infty$-category. It can be modeled as the $\infty$-localization of the nerve of $\sM$ at the weak equivalences. 

    The correspondence $\sM \mapsto \sM_{\infty}$ is (pseudo) $2$-functorial in the sense that:
    \begin{itemize}[leftmargin=0.6cm]
        \item a functor $F \colon \sM \to \sN$ between relative categories preserving weak equivalences induces a functor $F_{\infty} \colon \sM_{\infty} \to \sN_{\infty}$ between underlying $\infty$-categories,
        \item  a natural transformation $\alpha \colon F \Rightarrow G$ between functors preserving weak equivalences induces a natural transformation $\alpha_{\infty} \colon F_{\infty} \to G_{\infty}$,
    \end{itemize}
    in such a way which is compatible with composition of functors and natural transformations.
\end{notation}

\begin{ex}
    If $\sM$ is a model category, we also denote by $\sM_{\infty}$ the underlying $\infty$-category of its underlying relative category. 
\end{ex}

\begin{rem} \label{rem:liftQE}
    If $\alpha \colon F \Rightarrow G$ is a natural transformation of functors $\sM \to \sN$ between relative categories that preserve weak equivalences, then $\alpha$ is levelwise a weak equivalence in $\sN$ if and only if the induced transformation $\alpha_\infty\colon F_\infty\Rightarrow G_\infty$ is a natural isomorphism.
    
    It follows easily that, if $F\colon \sM\to \sN$ is a (right) Quillen functor between model categories preserving all weak equivalences, then $F$ is a Quillen equivalence if and only if the induced functor $F_\infty\colon \sM_\infty\to \sN_\infty$ is an equivalence of $\infty$-categories.
\end{rem}

Since the functor $\Cat(\Ex^2N)$ preserves all weak equivalences by \cref{thm:Quillenpairs}, it induces a functor between underlying $\infty$-categories 
\[
(\Cat(\Ex^2 N))_\infty\colon (\DblCat_\proj)_\infty \to (\Cat(\sSet)_\proj)_\infty.
\]
To prove \cref{thm:quilleneq}, it suffices to prove that this functor is an equivalence of $\infty$-categories. In order to do so, we replace the functor $\Ex^2 N$ with an equivalent, more convenient functor $i^*_\Simp\colon \Cat\to \sSet$, which admits a left adjoint that also preserves pullbacks. We recall the definition of these functors from \cite[\textsection 4.1.2]{Cisinski}. 

\begin{defn}
Consider the functor $i_{\Simp} \colon \sSet \to \Cat$ sending a simplicial set $X$ to its category of elements $\Simp/X$. The functor $i_\Simp$ admits a right adjoint given by the functor $i_{\Simp}^* \colon \Cat \to \sSet$ sending a category $\sC$ to the simplicial set 
\[ i_\Simp^* \sC\colon \Simpop\to \Set, \quad
    [k] \mapsto \Cat(\Simp/[k],\sC).
\] 
\end{defn}

We now show that the functors $\Ex^2 N$ and $i^*_\Simp$ are related by a zig-zag of levelwise weak equivalences. As we have already compare $\Ex^2N$ with the nerve functor $N$ in \cref{natbeta}, it remains to compare $N$ with $i^*_\Simp$.

\begin{prop} \label{natalpha}
    There is a natural transformations of functors $\Cat\to \Cat$
    \[ \alpha\colon i_\Simp N\Rightarrow \id_\Cat, \]
    which is levelwise a weak equivalence in $\Cat_\Thom$. Together with the unit $\eta\colon \id_\sSet\Rightarrow i_{\Simp}^*i_{\Simp}$ of the adjunction $i_\Simp\dashv i^*_\Simp$, it induces a natural transformation of functors $\Cat\to \sSet$
    \[ \gamma\colon N \xRightarrow{\eta N} i_{\Simp}^*i_{\Simp}N \xRightarrow{i_{\Simp}^*\alpha} i_{\Simp}^*, \]
    which is levelwise a weak equivalence in $\sSet_\Kan$. 
\end{prop}

\begin{proof}
    The first statement is an application of \cite[Theorem 4.1.26]{Cisinski}, in the case where $A=\Simp$ and $i\colon \Simp\to \Cat$ is the canonical inclusion, using \cite[Example 4.1.29]{Cisinski}.

    The second statement follows directly from the fact that the category $\Simp$ is a test category by \cite[Proposition 1.5.13]{Maltsiniotis}. Indeed, by definition of a test category (see \cite[\textsection1.3.7]{Maltsiniotis}), we have that the unit $\eta\colon \id_\sSet\Rightarrow i_{\Simp}^*i_{\Simp}$ is levelwise a weak equivalence in $\sSet_\Kan$, and that the functor $i^*_\Simp\colon \Cat_\Thom\to \sSet_\Kan$ preserves all weak equivalences. Hence, combining with the first statement, the desired composite $\gamma\colon N\Rightarrow i^*_\Simp$ is levelwise a weak equivalence in~$\sSet_\Kan$.
\end{proof}

Combining \cref{natbeta,natalpha}, we get the following result.

\begin{cor} \label{cor:zigzag1}
    There are natural transformations of functors $\Cat\to \sSet$
    \[
    \Ex^2N \xLeftarrow{\beta^2N} N \xRightarrow{\gamma}i_{\Simp}^*,
    \]
    which are levelwise a weak equivalence in $\sSet_\Kan$.
\end{cor}

We now aim to deduce that this induces a zig-zag of levelwise weak equivalences between the induced functors $\Cat(\Ex^2 N)$ and $\Cat(i^*_\Simp)$. To achieve this, we will make use of the following lemma. 

\begin{lemma} \label{lemma:ptwisewe}
    Let $\sM$ and $\sN$ be model categories, $F,G\colon \sM\to \sN$ be pullback-preserving functors, and $\alpha\colon F\Rightarrow G$ be a natural transformation. If $\alpha$ is levelwise a weak equivalence in $\sN$, then the induced natural transformation of functors $\Cat(\sM)\to \Cat(\sN)$ between categories of internal categories to $\sM$ and $\sN$
    \[
    \Cat(\alpha)\colon \Cat(F) \Rightarrow \Cat(G)
    \]
    is levelwise a weak equivalence in $\Cat(\sN)$. Here an internal functor $\bA\to \bB$ in $\Cat(\sN)$ is a weak equivalence if and only if, for every $n\geq 0$, the induced map 
    \[ \bA_1\times_{\bA_0}\ldots \times_{\bA_0}\bA_1\to \bB_1\times_{\bB_0}\ldots \times_{\bB_0}\bB_1\]
    is a weak equivalence in $\sN$.
\end{lemma}

\begin{proof}
    Let $\bA$ be an internal category to $\sM$. For every $n\geq 0$, we need to show that the map induced by the maps $\alpha_{\bA_0}$ and $\alpha_{\bA_1}$
    \[ F(\bA_1)\times_{F(\bA_0)}\ldots \times_{F(\bA_0)}F(\bA_1)\to G(\bA_1)\times_{G(\bA_0)}\ldots \times_{G(\bA_0)}G(\bA_1) \]
    is a weak equivalence in $\sN$. Since $F,G$ preserve pullbacks and $\alpha$ is natural, the above map can be identified with the map 
    \[ \alpha_{\bA_1\times_{\bA_0}\ldots \times_{\bA_0}\bA_1}\colon F(\bA_1\times_{\bA_0}\ldots \times_{\bA_0}\bA_1)\to G(\bA_1\times_{\bA_0}\ldots \times_{\bA_0}\bA_1) \]
    which is a weak equivalence in $\sN$, by assumption. 
\end{proof}

\begin{rem} \label{rem:weininternal}
    Note that the weak equivalences in $\Cat(\sN)$, as defined in the above lemma, coincide in the case where $\sN=\Cat_\Thom,\,\sSet_\Kan$ with the weak equivalences of the model structures $\DblCat_\proj$ and $\Cat(\sSet)_\proj$, as these are right-induced along the nerves from $\Cat^{\Simpop}_\proj$ and $\sSet^{\Simpop}_\proj$, respectively.
\end{rem}

The following result is then a direct application of \cref{lemma:ptwisewe} to the natural transformations from \cref{cor:zigzag1}, using \cref{rem:weininternal}.

\begin{prop}\label{cor:zigzag}
    The natural transformations of functors $\DblCat\to \Cat(\sSet)$
    \[
    \Cat(\Ex^2N)  \xLeftarrow{\Cat(\beta^2N)}\Cat(N) \xRightarrow{\Cat(\gamma)} \Cat(i_{\Simp}^*),
    \]
    induced by the natural transformations from \cref{cor:zigzag1}, are levelwise a weak equivalence in $\Cat(\sSet)_\proj$. 
\end{prop}

 \begin{rem}
     As a consequence of \cref{cor:zigzag}, we have that the functors 
    \[ \Cat(N),\Cat(i_{\Simp}^*)\colon \DblCat_\proj\to \Cat(\sSet)_\proj \]
    also preserve (and reflect) all weak equivalences, as $\Cat(\Ex^2N)$ does so by \cref{thm:Quillenpairs}. In particular, they induce functors between underlying $\infty$-categories.
 \end{rem}    

 As a direct consequence of \cref{cor:zigzag}, we get the following.

 \begin{cor} \label{cor:htypeqfunctors}
     The functors between underlying $\infty$-categories
    \[
    (\DblCat_\proj)_\infty\to (\Cat(\sSet)_\proj)_\infty
    \]
    induced by $\Cat(\Ex^2N)$, $\Cat(N)$, and $\Cat(i_{\Simp}^*)$ are all naturally equivalent.
 \end{cor}
 
 The proof of \cref{thm:quilleneq} is then completed by the following proposition.

\begin{prop} \label{prop:HoiDeltaeq}
    The functor $\Cat(i^*_\Simp)\colon \DblCat\to \Cat(\sSet)$ induces an equivalence of $\infty$-categories
    \[ (\Cat(i_\Simp^*))_\infty\colon (\DblCat_\proj)_\infty \to (\Cat(\sSet)_\proj)_\infty. \]
\end{prop}

\begin{proof}
    First recall from \cite[Corollary 3.2.13]{Cisinski} that the left adjoint $i_{\Simp} \colon \sSet \to \Cat$ preserves pullbacks. Hence the composites $i_{\Simp}i_{\Simp}^* \colon \Cat \to \Cat$ and $i_{\Simp}^*i_{\Simp} \colon \sSet \to \sSet$ also preserve pullbacks.
    
    Now, since $\Simp$ is a test category by \cite[Proposition 1.5.13]{Maltsiniotis}, by definition (see \cite[\textsection 1.3.7]{Maltsiniotis}), the unit and counit of the adjunction 
    \begin{tz}
\node[](A) {$\sSet$};
\node[right of=A,xshift=.9cm](B) {$\Cat$};
\draw[->] ($(A.east)+(0,5pt)$) to node[above,la]{$i_{\Simp}$} ($(B.west)+(0,5pt)$);
\draw[->] ($(B.west)-(0,5pt)$) to node[below,la]{$i_\Simp^*$} ($(A.east)-(0,5pt)$);
\node[la] at ($(A.east)!0.5!(B.west)$) {$\bot$};
\end{tz}
are levelwise a weak equivalence in $\sSet_\Kan$ and $\Cat_\Thom$, respectively. Using \cref{lemma:ptwisewe,rem:weininternal}, it follows that the unit and counit of the adjunction
\begin{tz}
\node[](A) {$\Cat(\sSet)$};
\node[right of=A,xshift=1.6cm](B) {$\DblCat$};
\draw[->] ($(A.east)+(0,5pt)$) to node[above,la]{$\Cat(i_{\Simp})$} ($(B.west)+(0,5pt)$);
\draw[->] ($(B.west)-(0,5pt)$) to node[below,la]{$\Cat(i_\Simp^*)$} ($(A.east)-(0,5pt)$);
\node[la] at ($(A.east)!0.5!(B.west)$) {$\bot$};
\end{tz}
are also levelwise a weak equivalence in $\Cat(\sSet)_\proj$ and $\DblCat_\proj$, respectively. Hence the induced adjunction between underlying $\infty$-categories is an equivalence.
\end{proof}

\begin{proof}[Proof of \cref{thm:quilleneq}]
    To show that the Quillen pair $\Cat(c\Sd^2)\dashv \Cat(\Ex^2 N)$ is a Quillen equivalence, it suffices to show that the induced functor between underlying $\infty$-categories $(\Cat(\Ex^2 N))_\infty\colon (\DblCat_\proj)_\infty \to (\Cat(\sSet)_\proj)_\infty$ is an equivalence of $\infty$-categories. This follows directly from \cref{cor:htypeqfunctors,prop:HoiDeltaeq}.
\end{proof}

As a consequence, we have the following result.

\begin{theorem} \label{thm:allQE} 
    We have a commutative square of Quillen equivalences
    \begin{tz}
    \node[](1) {$\sSet^{\Simpop}_\proj$}; 
    \node[below of=1,yshift=-.5cm](2) {$\Cat(\sSet)_\proj$}; 
    \node[right of=1,xshift=2.2cm](3) {$\Cat^{\Simpop}_\proj$}; 
    \node[below of=3,yshift=-.5cm](4) {$\DblCat_\proj$};
\punctuation{4}{,};

    \draw[->] ($(1.east)+(0,5pt)$) to node[above,la]{$(c\Sd^2)_*$} ($(3.west)+(0,5pt)$);
\draw[->] ($(3.west)-(0,5pt)$) to node[below,la]{$(\Ex^2 N)_*$} ($(1.east)-(0,5pt)$);
\node[la] at ($(1.east)!0.5!(3.west)$) {$\bot$};
\draw[->] ($(2.east)+(0,5pt)$) to node[above,la]{$\Cat(c\Sd^2)$} ($(4.west)+(0,5pt)$);
\draw[->] ($(4.west)-(0,5pt)$) to node[below,la]{$\Cat(\Ex^2 N)$} ($(2.east)-(0,5pt)$);
\node[la] at ($(2.east)!0.5!(4.west)$) {$\bot$};

\draw[->] ($(3.south)-(5pt,0)$) to node[left,la]{$\ch$} ($(4.north)-(5pt,0)$); 
\draw[->] ($(4.north)+(5pt,0)$) to node[right,la]{$\Nh$} ($(3.south)+(5pt,0)$); 
\node[la] at ($(3.south)!0.5!(4.north)$) {\rotatebox{90}{$\bot$}};
\draw[->] ($(1.south)-(5pt,0)$) to node[left,la]{$c^\Simp$} ($(2.north)-(5pt,0)$); 
\draw[->] ($(2.north)+(5pt,0)$) to node[right,la]{$N^\Simp$} ($(1.south)+(5pt,0)$); 
\node[la] at ($(1.south)!0.5!(2.north)$) {\rotatebox{90}{$\bot$}};
\end{tz}
    where all model structures are right-induced from $\sSet^{\Simpop}_\proj$. 
\end{theorem}

\begin{proof}
    The left, top, and bottom Quillen pairs are Quillen equivalences by \cref{QEHorel,QEprojThomason,thm:quilleneq}. Hence, by $2$-out-of-$3$ for Quillen equivalences, we get that the right Quillen pair is also a Quillen equivalence.
\end{proof}

Finally, as the double nerve preserves all weak equivalences by \cref{prop:charwe}, it defines a functor between underlying $\infty$-categories, which can be seen to be an equivalence by \cref{natbetadouble,thm:allQE}. 

\begin{cor}
    The double nerve functor $\bN\colon \DblCat\to \sSet^{\Simpop}$ induces an equivalence of $\infty$-categories 
    \[ \bN_\infty\colon (\DblCat_\proj)_\infty\to (\sSet^{\Simpop}_\proj)_\infty. \]
\end{cor}

\section{Double categorical model of \texorpdfstring{$(\infty,1)$}{(infinity,1)}-categories}

We now aim to localize the model structure on $\DblCat$ to obtain a model of $(\infty,1)$-categories. To achieve this, we first perform a localization in \cref{sec:Segal} with respect to the image of the Segal maps in $\sSet^{\Simpop}$, resulting in a double categorical model of Segal spaces. Then, in \cref{sec:CSS}, we further localize with respect to the image of the completeness map in $\sSet^{\Simpop}$, yielding the desired double categorical model of $(\infty,1)$-categories. 

\subsection{Segal localization} \label{sec:Segal}

We now want to localize the model structure $\DblCat_\proj$ in order to obtain a model for Segal spaces. For this, recall that the Quillen equivalent model $\sSet^{\Simpop}_\proj$ is localized at the \emph{spine inclusions} to obtain such a model. 

\begin{defn}
    For $n\geq 2$, let $\Simp[1]\amalg^h_{[0]} \ldots \amalg^h_{[0]} \Simp[1]$ denote the homotopy colimit in $\sSet^{\Simpop}_\proj$ of $n$ copies of $\Simp[1]$ under $\Simp[0]$, as illustrated by the following diagram in $\Set^{\Simpop}$, seen as a diagram in $\sSet^{\Simpop}$ through the canonical inclusion $\iota\colon \Set^{\Simpop}\to \sSet^{\Simpop}$.
\begin{tz}
    \node[](1) {$\Simp[1]$}; 
    \node[above right of=1,xshift=.5cm](2) {$\Simp[0]$};
    \node[below right of=2,xshift=.5cm](3) {$\Simp[1]$}; 
    \node[above right of=3,xshift=.5cm](4){$\Simp[0]$}; 
    \node[below right of=4,white,xshift=.5cm](5) {$\Simp[1]$}; 
    \node[below right of=4,xshift=.5cm] {$\ldots$}; 
    \node[above right of=5,xshift=.5cm](6){$\Simp[0]$};
    \node[below right of=6,xshift=.5cm](7) {$\Simp[1]$}; 

    \draw[->](2) to node[above,la]{$d^0$} (1);
    \draw[->](2) to node[above,la,xshift=3pt]{$d^1$} (3);
    \draw[->](4) to node[above,la]{$d^0$} (3);
    \draw[->](4) to node[above,la,xshift=3pt]{$d^1$} (5);
    \draw[->](6) to node[above,la]{$d^0$} (5);
    \draw[->](6) to node[above,la,xshift=3pt]{$d^1$} (7);
\end{tz}

The \textbf{spine inclusion} is the map $\Simp[1]\amalg^h_{[0]} \ldots \amalg^h_{[0]} \Simp[1]\to \Simp[n]$ induced by the maps in~$\Simp$
\[ \rho_i\colon [1]\to [n], \quad 0\mapsto i-1,\, 1\mapsto i, \]
for $1\leq i\leq n$. 
\end{defn}

\begin{notation} \label{loc:Seg}
    Let $\Seg$ denote the set of morphisms in $\sSet^{\Simpop}_\proj$
\[ \Seg\coloneqq  \{ (\Simp[1]\amalg^h_{[0]} \ldots \amalg^h_{[0]} \Simp[1]\to \Simp[n])\boxtimes \Simp[0]\mid n\geq 2\}.  \]
We also denote by $\Seg$ the image of $\Seg$ under the derived functors of the left Quillen functors from \cref{thm:Quillenpairs}. Hence we have sets of morphisms 
\begin{itemize}[leftmargin=0.6cm]
    \item $\Seg\coloneqq \bL( c^h (c \Sd^2)_*)(\Seg)$ in $\DblCat_\proj$, 
    \item $\Seg\coloneqq \bL c^h(\Seg)$ in $\Cat^{\Simpop}_\proj$,
    \item $\Seg\coloneqq \bL c^\Simp(\Seg)$ in $\Cat(\sSet)_\proj$.
\end{itemize} 
\end{notation}

Since the model structures $\DblCat_\proj$, $\Cat^{\Simpop}_\proj$, $\Cat(\sSet)_\proj$, and $\sSet^{\Simpop}_\proj$ are left proper and combinatorial by \cref{Thomproj,thm:dblcatleftproper,projKan,HorelMS}, by \cref{thm:locexist}, we can localize them at the sets $\Seg$.

\begin{prop}
    The $\Seg$-localized model structures $\DblCat_\Seg$, $\Cat^{\Simpop}_\Seg$, $\Cat(\sSet)_\Seg$, and $\sSet^{\Simpop}_\Seg$ exist. 
\end{prop}

\begin{rem} \label{rem:explicitsetSeg}
    As $\Simp[n]$ is cofibrant in $\sSet^{\Simpop}_\proj$, for all $n\geq 0$, the model structures $\DblCat_\Seg$, $\Cat^{\Simpop}_\Seg$, and $\Cat(\sSet)_\Seg$ coincide with the localizations
    \begin{itemize}[leftmargin=0.6cm]
    \item of $\DblCat_\proj$ at the set $\{ ([1]\amalg_{[0]}^h\ldots \amalg_{[0]}^h [1]\to [n])\boxtimes [0]\mid n\geq 0\}$, 
    \item of $\Cat^{\Simpop}_\proj$ at the set $\{ (\Simp[1]\amalg^h_{[0]} \ldots \amalg^h_{[0]} \Simp[1]\to \Simp[n])\boxtimes [0]\mid n\geq 0\}$,
    \item of $\Cat(\sSet)_\proj$ at the set $\{ ([1]\amalg_{[0]}^h\ldots \amalg_{[0]}^h [1]\to [n])\boxtimes \Simp[0]\mid n\geq 0\}$.
\end{itemize} 
   We give an explicit description of the homotopy colimit $([1]\amalg_{[0]}^h\ldots \amalg_{[0]}^h [1])\boxtimes [0]$ in $\DblCat_\proj$ in \cref{sec:application}.
\end{rem}

By applying \cref{thm:RIMS} to the commutative square of Quillen equivalences from \cref{thm:allQE}, we directly get the following. 

\begin{theorem} \label{squareSeg}
    We have a commutative square of Quillen equivalences
    \begin{tz}
    \node[](1) {$\sSet^{\Simpop}_\Seg$}; 
    \node[below of=1,yshift=-.5cm](2) {$\Cat(\sSet)_\Seg$}; 
    \node[right of=1,xshift=2.1cm](3) {$\Cat^{\Simpop}_\Seg$}; 
    \node[below of=3,yshift=-.5cm](4) {$\DblCat_\Seg$};
\punctuation{4}{,};

    \draw[->] ($(1.east)+(0,5pt)$) to node[above,la]{$(c\Sd^2)_*$} ($(3.west)+(0,5pt)$);
\draw[->] ($(3.west)-(0,5pt)$) to node[below,la]{$(\Ex^2 N)_*$} ($(1.east)-(0,5pt)$);
\node[la] at ($(1.east)!0.5!(3.west)$) {$\bot$};
\draw[->] ($(2.east)+(0,5pt)$) to node[above,la]{$\Cat(c\Sd^2)$} ($(4.west)+(0,5pt)$);
\draw[->] ($(4.west)-(0,5pt)$) to node[below,la]{$\Cat(\Ex^2 N)$} ($(2.east)-(0,5pt)$);
\node[la] at ($(2.east)!0.5!(4.west)$) {$\bot$};

\draw[->] ($(3.south)-(5pt,0)$) to node[left,la]{$\ch$} ($(4.north)-(5pt,0)$); 
\draw[->] ($(4.north)+(5pt,0)$) to node[right,la]{$\Nh$} ($(3.south)+(5pt,0)$); 
\node[la] at ($(3.south)!0.5!(4.north)$) {\rotatebox{90}{$\bot$}};
\draw[->] ($(1.south)-(5pt,0)$) to node[left,la]{$c^\Simp$} ($(2.north)-(5pt,0)$); 
\draw[->] ($(2.north)+(5pt,0)$) to node[right,la]{$N^\Simp$} ($(1.south)+(5pt,0)$); 
\node[la] at ($(1.south)!0.5!(2.north)$) {\rotatebox{90}{$\bot$}};
\end{tz}
    where all model structures are right-induced. 
\end{theorem}

Similarly to the projective case, we can characterize the weak equivalences in $\DblCat_\Seg$ through their double nerve. 

\begin{prop} 
    A double functor $F$ is a weak equivalence in $\DblCat_\Seg$ if and only if its double nerve $\bN F$ is a weak equivalence in $\sSet^{\Simpop}_\Seg$. 
\end{prop}

\begin{proof}
    This follows directly from \cref{natbetadouble} and the fact that the weak equivalences in $\DblCat_\Seg$ are induced from those of $\sSet^{\Simpop}_\Seg$ along the right adjoint functor $(\Ex^2N)_* N^h$ by \cref{squareSeg}.
\end{proof}

Finally, as the double nerve preserves all weak equivalences by the above result, it defines a functor between underlying $\infty$-categories, which can be seen to be an equivalence by \cref{natbetadouble,squareSeg}.

\begin{cor}
    The double nerve functor $\bN\colon \DblCat\to \sSet^{\Simpop}$ induces an equivalence of $\infty$-categories 
    \[ \bN_\infty\colon (\DblCat_\Seg)_\infty\to (\sSet^{\Simpop}_\Seg)_\infty. \]
\end{cor}

\subsection{Complete Segal localization} \label{sec:CSS}

We now aim to further localize $\DblCat_\Seg$ in order to obtain a model for complete Segal spaces, and so $(\infty,1)$-categories. For this, recall that the Quillen equivalent model $\sSet^{\Simpop}_\Seg$ is localized at the \emph{completeness map} to obtain such a model. 

\begin{defn}
    Let $\Simp[0]\amalg^h_{\Simp[1]} \Simp[3]\amalg^h_{\Simp[1]}\Simp[0]$ denote the homotopy colimit in $\sSet^{\Simpop}_\proj$ of the following diagram in $\Set^{\Simpop}$, seen as a diagram in $\sSet^{\Simpop}$ through the canonical inclusion $\iota\colon \Set^{\Simpop}\to \sSet^{\Simpop}$. 
    \begin{tz}
    \node[](1) {$\Simp[0]$}; 
    \node[above right of=1,xshift=.5cm](2) {$\Simp[1]$};
    \node[below right of=2,xshift=.5cm](3) {$\Simp[3]$}; 
    \node[above right of=3,xshift=.5cm](4){$\Simp[1]$}; 
    \node[below right of=4,xshift=.5cm](5) {$\Simp[0]$};  

    \draw[->](2) to node[above,la]{$!$} (1);
    \draw[->](2) to node[above,la,xshift=7pt]{$d^3d^1$} (3);
    \draw[->](4) to node[above,la,xshift=-4pt]{$d^0d^1$} (3);
    \draw[->](4) to node[above,la]{$!$} (5);
\end{tz}

    The \textbf{completeness map} is the unique map $\Simp[0]\amalg^h_{\Simp[1]} \Simp[3]\amalg^h_{\Simp[1]}\Simp[0]\to \Simp[0]$.
\end{defn}

\begin{notation}
    Let $\CSS$ denote the set of morphisms in $\sSet^{\Simpop}_\proj$
\[\CSS\coloneqq \Seg\cup \{ (\Simp[0]\amalg^h_{\Simp[1]} \Simp[3]\amalg^h_{\Simp[1]}\Simp[0]\to \Simp[0])\boxtimes \Simp[0]\}. \]
We also denote by $\CSS$ the image of $\CSS$ under the derived functors of the left Quillen functors from \cref{thm:Quillenpairs}. Hence we have sets of morphisms 
\begin{itemize}[leftmargin=0.6cm]
    \item $\CSS\coloneqq \bL( c^h (c \Sd^2)_*)(\CSS)$ in $\DblCat_\proj$, 
    \item $\CSS\coloneqq \bL c^h(\CSS)$ in $\Cat^{\Simpop}_\proj$,
    \item $\CSS\coloneqq \bL c^\Simp(\CSS)$ in $\Cat(\sSet)_\proj$.
\end{itemize} 
\end{notation} 

As before, since the model structures $\DblCat_\proj$, $\Cat^{\Simpop}_\proj$, $\Cat(\sSet)_\proj$, and $\sSet^{\Simpop}_\proj$ are left proper and combinatorial by \cref{Thomproj,thm:dblcatleftproper,projKan,HorelMS}, by \cref{thm:locexist}, we can localize them at the sets $\CSS$.

\begin{prop}
    The $\CSS$-localized model structures $\DblCat_\CSS$, $\Cat^{\Simpop}_\CSS$, $\Cat(\sSet)_\CSS$, and $\sSet^{\Simpop}_\CSS$ exist. 
\end{prop}

\begin{rem} \label{rem:explicitsetCSS}
    As $\Simp[n]$ is cofibrant in $\sSet^{\Simpop}_\proj$, for all $n\geq 0$, the model structures $\DblCat_\CSS$, $\Cat^{\Simpop}_\CSS$, and $\Cat(\sSet)_\CSS$ coincide with the localizations
    \begin{itemize}[leftmargin=0.6cm]
    \item of $\DblCat_\proj$ at the set $\Seg\cup \{([0]\amalg_{[1]}^h [3]\amalg_{[1]} [0]\to [0])\boxtimes [0]\}$, 
    \item of $\Cat^{\Simpop}_\proj$ at the set $\Seg\cup \{ (\Simp[0]\amalg^h_{\Simp[1]} \Simp[3]\amalg^h_{\Simp[1]}\Simp[0]\to \Simp[0])\boxtimes\Simp[0] \}$,
    \item of $\Cat(\sSet)_\proj$ at the set $\Seg\cup \{([0]\amalg_{[1]}^h [3]\amalg_{[1]} [0]\to [0])\boxtimes [0]\}$,
\end{itemize} 
    where the localizations with respect to $\Seg$ can be computed as in \cref{rem:explicitsetSeg}. We give an explicit description of the homotopy colimit $([0]\amalg_{[1]}^h [3]\amalg_{[1]} [0])\boxtimes [0]$ in $\DblCat_\proj$ in \cref{sec:application}.
\end{rem}

By applying \cref{thm:RIMS} to the commutative square of Quillen equivalences from \cref{thm:allQE}, we directly get the following.

\begin{theorem} \label{squareCSS}
    We have a commutative square of Quillen equivalences
    \begin{tz}
    \node[](1) {$\sSet^{\Simpop}_\CSS$}; 
    \node[below of=1,yshift=-.5cm](2) {$\Cat(\sSet)_\CSS$}; 
    \node[right of=1,xshift=2.2cm](3) {$\Cat^{\Simpop}_\CSS$}; 
    \node[below of=3,yshift=-.5cm](4) {$\DblCat_\CSS$};
\punctuation{4}{,};

    \draw[->] ($(1.east)+(0,5pt)$) to node[above,la]{$(c\Sd^2)_*$} ($(3.west)+(0,5pt)$);
\draw[->] ($(3.west)-(0,5pt)$) to node[below,la]{$(\Ex^2 N)_*$} ($(1.east)-(0,5pt)$);
\node[la] at ($(1.east)!0.5!(3.west)$) {$\bot$};
\draw[->] ($(2.east)+(0,5pt)$) to node[above,la]{$\Cat(c\Sd^2)$} ($(4.west)+(0,5pt)$);
\draw[->] ($(4.west)-(0,5pt)$) to node[below,la]{$\Cat(\Ex^2 N)$} ($(2.east)-(0,5pt)$);
\node[la] at ($(2.east)!0.5!(4.west)$) {$\bot$};

\draw[->] ($(3.south)-(5pt,0)$) to node[left,la]{$\ch$} ($(4.north)-(5pt,0)$); 
\draw[->] ($(4.north)+(5pt,0)$) to node[right,la]{$\Nh$} ($(3.south)+(5pt,0)$); 
\node[la] at ($(3.south)!0.5!(4.north)$) {\rotatebox{90}{$\bot$}};
\draw[->] ($(1.south)-(5pt,0)$) to node[left,la]{$c^\Simp$} ($(2.north)-(5pt,0)$); 
\draw[->] ($(2.north)+(5pt,0)$) to node[right,la]{$N^\Simp$} ($(1.south)+(5pt,0)$); 
\node[la] at ($(1.south)!0.5!(2.north)$) {\rotatebox{90}{$\bot$}};
\end{tz}
    where all model structures are right-induced. 
\end{theorem}

We again have the following characterization of the weak equivalences in $\DblCat_\CSS$ through their double nerve. 

\begin{prop}
    A double functor $F$ is a weak equivalence in $\DblCat_\CSS$ if and only if its double nerve $\bN F$ is a weak equivalence in $\sSet^{\Simpop}_\CSS$. 
\end{prop}

\begin{proof}
    This follows directly from \cref{natbetadouble} and the fact that the weak equivalences in $\DblCat_\CSS$ are induced from those of $\sSet^{\Simpop}_\CSS$ along the right adjoint functor $(\Ex^2N)_* N^h$ by \cref{squareCSS}.
\end{proof}

Finally, as the double nerve preserves all weak equivalences by the above result, it defines a functor between underlying $\infty$-categories, which can be seen to be an equivalence by \cref{natbetadouble,squareCSS}.

\begin{cor}
    The double nerve functor $\bN\colon \DblCat\to \sSet^{\Simpop}$ induces an equivalence of $\infty$-categories 
    \[ \bN_\infty\colon (\DblCat_\CSS)_\infty\to (\sSet^{\Simpop}_\CSS)_\infty. \]
\end{cor}

\section{Double categorical model of \texorpdfstring{$\infty$}{infinity}-groupoids}

We now aim to localize the model structure on $\DblCat$ to obtain a model for $\infty$-groupoids. In \cref{sec:gpdloc}, we perform this localization at a relevant set of maps and show that the resulting localized model structure is Quillen equivalent to the Thomason model structure on $\Cat$. However, this is not the first double-categorical model for $\infty$-groupoids to appear in the literature: Fiore--Paoli construct in \cite{FP} a Thomason-like model structure on $\DblCat$ (and more generally on $n$-fold categories) that is Quillen equivalent to the Kan--Quillen model structure on $\sSet$. In \cref{sec:GpdvsThom}, we show that our localized model structure and their Thomason model structure have the same weak equivalences, but that the identity functor at $\DblCat$ does not induce a Quillen equivalence between them.

\subsection{Groupoidal localization} \label{sec:gpdloc}

We now want to localize the model structure $\DblCat_\proj$ in order to obtain a model of $\infty$-groupoids. For this, we localize at the following set of maps. 

\begin{notation}
    Let $\Gpd$ denote the set of morphisms in $\sSet^{\Simpop}_\proj$
\[\Gpd \coloneqq \{ (\Simp[n]\to \Simp[0])\boxtimes \Simp[0]\mid n\geq 0\}. \]
We also denote by $\Gpd$ the image of $\Gpd$ under the derived functors of the left Quillen functors from \cref{thm:Quillenpairs}. Hence we have sets of morphisms 
\begin{itemize}[leftmargin=0.6cm]
    \item $\Gpd\coloneqq \bL( c^h (c \Sd^2)_*)(\Gpd)$ in $\DblCat_\proj$, 
    \item $\Gpd\coloneqq \bL c^h(\Gpd)$ in $\Cat^{\Simpop}_\proj$,
    \item $\Gpd\coloneqq \bL c^\Simp(\Gpd)$ in $\Cat(\sSet)_\proj$.
\end{itemize} 
\end{notation}

As before, since the model structures $\DblCat_\proj$, $\Cat^{\Simpop}_\proj$, $\Cat(\sSet)_\proj$, and $\sSet^{\Simpop}_\proj$ are left proper and combinatorial by \cref{Thomproj,thm:dblcatleftproper,projKan,HorelMS}, by \cref{thm:locexist}, we can localize them at the sets $\Gpd$.

\begin{prop}
    The $\Gpd$-localized model structures $\DblCat_\Gpd$, $\Cat^{\Simpop}_\Gpd$, $\Cat(\sSet)_\Gpd$, and $\sSet^{\Simpop}_\Gpd$ exist. 
\end{prop}

\begin{rem}
    As $\Simp[n]$ is cofibrant in $\sSet^{\Simpop}_\proj$, for all $n\geq 0$, the model structures $\DblCat_\Gpd$, $\Cat^{\Simpop}_\Gpd$, and $\Cat(\sSet)_\Gpd$ coincide with the localizations
    \begin{itemize}[leftmargin=0.6cm]
    \item of $\DblCat_\proj$ at the set $\{ ([n]\to [0])\boxtimes [0]\mid n\geq 0\}$, 
    \item of $\Cat^{\Simpop}_\proj$ at the set $\{ (\Simp[n]\to \Simp[0])\boxtimes [0]\mid n\geq 0\}$,
    \item of $\Cat(\sSet)_\proj$ at the set $\{ ([n]\to [0])\boxtimes \Simp[0]\mid n\geq 0\}$.
\end{itemize} 
\end{rem}

By applying \cref{thm:RIMS} to the commutative square of Quillen equivalences from \cref{thm:allQE}, we directly get the following.

\begin{theorem} \label{thm:gpdsquare}
    We have a commutative square of Quillen equivalences
    \begin{tz}
    \node[](1) {$\sSet^{\Simpop}_\Gpd$}; 
    \node[below of=1,yshift=-.5cm](2) {$\Cat(\sSet)_\Gpd$}; 
    \node[right of=1,xshift=2.2cm](3) {$\Cat^{\Simpop}_\Gpd$}; 
    \node[below of=3,yshift=-.5cm](4) {$\DblCat_\Gpd$};
\punctuation{4}{,};

    \draw[->] ($(1.east)+(0,5pt)$) to node[above,la]{$(c\Sd^2)_*$} ($(3.west)+(0,5pt)$);
\draw[->] ($(3.west)-(0,5pt)$) to node[below,la]{$(\Ex^2 N)_*$} ($(1.east)-(0,5pt)$);
\node[la] at ($(1.east)!0.5!(3.west)$) {$\bot$};
\draw[->] ($(2.east)+(0,5pt)$) to node[above,la]{$\Cat(c\Sd^2)$} ($(4.west)+(0,5pt)$);
\draw[->] ($(4.west)-(0,5pt)$) to node[below,la]{$\Cat(\Ex^2 N)$} ($(2.east)-(0,5pt)$);
\node[la] at ($(2.east)!0.5!(4.west)$) {$\bot$};

\draw[->] ($(3.south)-(5pt,0)$) to node[left,la]{$\ch$} ($(4.north)-(5pt,0)$); 
\draw[->] ($(4.north)+(5pt,0)$) to node[right,la]{$\Nh$} ($(3.south)+(5pt,0)$); 
\node[la] at ($(3.south)!0.5!(4.north)$) {\rotatebox{90}{$\bot$}};
\draw[->] ($(1.south)-(5pt,0)$) to node[left,la]{$c^\Simp$} ($(2.north)-(5pt,0)$); 
\draw[->] ($(2.north)+(5pt,0)$) to node[right,la]{$N^\Simp$} ($(1.south)+(5pt,0)$); 
\node[la] at ($(1.south)!0.5!(2.north)$) {\rotatebox{90}{$\bot$}};
\end{tz}
    where all model structures are right-induced. 
\end{theorem}

Similarly to the other cases, we have the following characterization of the weak equivalences in $\DblCat_\Gpd$ through their double nerve. 

\begin{prop} \label{cor:charweGpd}
    A double functor $F$ is a weak equivalence in $\DblCat_\Gpd$ if and only if its double nerve $\bN F$ is a weak equivalence in $\sSet^{\Simpop}_\Gpd$. 
\end{prop}

\begin{proof}
    This follows directly from \cref{natbetadouble} and the fact that the weak equivalences in $\DblCat_\Gpd$ are induced from those of $\sSet^{\Simpop}_\Gpd$ along the right adjoint functor $(\Ex^2N)_* N^h$ by \cref{thm:gpdsquare}.
\end{proof}

As the double nerve preserves all weak equivalences by the above remark, it defines a functor between underlying $\infty$-categories, which can be seen to be an equivalence by \cref{natbetadouble,squareSeg}.

\begin{cor}
    The double nerve functor $\bN\colon \DblCat\to \sSet^{\Simpop}$ induces an equivalence of $\infty$-categories 
    \[ \bN_\infty\colon (\DblCat_\Gpd)_\infty\to (\sSet^{\Simpop}_\Gpd)_\infty. \]
\end{cor}

The fact that these model structures are models of $\infty$-groupoids is justified by the following result. For this, recall that the functor $\cst\colon \sSet\to \sSet^{\Simpop}$ admits a right adjoint $(-)_0\colon \sSet^{\Simpop}\to \sSet$ given by evaluating an object in $\sSet^{\Simpop}$ at~$0$.

\begin{prop} \label{modelofgpd}
    The adjunction
    \begin{tz}
\node[](A) {$\sSet_\Kan$};
\node[right of=A,xshift=1.5cm](B) {$\sSet^{\Simpop}_{\Gpd}$};
\draw[->] ($(A.east)+(0,5pt)$) to node[above,la]{$\cst$} ($(B.west)+(0,5pt)$);
\draw[->] ($(B.west)-(0,5pt)$) to node[below,la]{$(-)_0$} ($(A.east)-(0,5pt)$);
\node[la] at ($(A.east)!0.5!(B.west)$) {$\bot$};
\end{tz}
    is a Quillen equivalence. 
\end{prop}

\begin{proof}
    It is straightforward to check that the functor $\cst\colon \sSet_\Kan\to \sSet^{\Simpop}_\Gpd$ preserves cofibrations and weak equivalences, and that it also reflects weak equivalences. Hence, to prove that $\cst\dashv (-)_0$ is a Quillen equivalence, by \cite[Lemma 3.1]{EG} it remains to show that the plain counit at a fibrant object is a weak equivalence. 
    
    Given a fibrant object $X$ in~$\sSet^{\Simpop}_\Gpd$, the counit at $X$ is given by the map $\cst(X_0)\to X$ whose $n$-th component, for $n\geq 0$, is the map $X_0\to X_n$ induced by the unique map $[n]\to [0]$ in $\Simp$. Since $X$ is fibrant in $\sSet^{\Simpop}_\Gpd$, it is local with respect to $\Simp[n]\to \Simp[0]$, and so the map $X_0\to X_n$ is a weak equivalence in $\sSet_\Kan$. This shows that $\cst(X_0)\to X$ is a weak equivalence in $\sSet^{\Simpop}_\Gpd$, as desired. 
\end{proof}

Finally, we prove that the model category $\DblCat_\Gpd$ is Quillen equivalent to the Thomason model structure $\Cat_\Thom$ through the adjunction $\bV\dashv \bfV$. 

\begin{lemma} \label{cstvsV}
    We have a commutative square of adjunctions
    \begin{tz}
    \node[](1) {$\sSet$}; 
    \node[below of=1,yshift=-.5cm](2) {$\Cat$}; 
    \node[right of=1,xshift=1.5cm](3) {$\sSet^{\Simpop}$}; 
    \node[below of=3,yshift=-.5cm](4) {$\DblCat$};
\punctuation{4}{.};

    \draw[->] ($(1.east)+(0,5pt)$) to node[above,la]{$\cst$} ($(3.west)+(0,5pt)$);
\draw[->] ($(3.west)-(0,5pt)$) to node[below,la]{$(-)_0$} ($(1.east)-(0,5pt)$);
\node[la] at ($(1.east)!0.5!(3.west)$) {$\bot$};
\draw[->] ($(2.east)+(0,5pt)$) to node[above,la]{$\bV$} ($(4.west)+(0,5pt)$);
\draw[->] ($(4.west)-(0,5pt)$) to node[below,la]{$\bfV$} ($(2.east)-(0,5pt)$);
\node[la] at ($(2.east)!0.5!(4.west)$) {$\bot$};

\draw[->] ($(3.south)-(5pt,0)$) to node[left,la]{$\ch(c\Sd^2)_*$} ($(4.north)-(5pt,0)$); 
\draw[->] ($(4.north)+(5pt,0)$) to node[right,la]{$(\Ex^2 N)_*\Nh$} ($(3.south)+(5pt,0)$); 
\node[la] at ($(3.south)!0.5!(4.north)$) {\rotatebox{90}{$\bot$}};
\draw[->] ($(1.south)-(5pt,0)$) to node[left,la]{$c\Sd^2$} ($(2.north)-(5pt,0)$); 
\draw[->] ($(2.north)+(5pt,0)$) to node[right,la]{$\Ex^2 N$} ($(1.south)+(5pt,0)$); 
\node[la] at ($(1.south)!0.5!(2.north)$) {\rotatebox{90}{$\bot$}};
\end{tz}
\end{lemma}

\begin{proof}
    We show that the square of left adjoints commutes. Since $\sSet$ is a presheaf category, it is enough to show that the composite of left adjoints agree on representables. Let $\Simp[k]$ be a representable in $\sSet$, for $k\geq 0$. Then, using the definitions of the box products and \cref{chonboxprod}, we have isomorphisms in $\DblCat$
    \[ c^h(c\Sd^2)_*(\cst\Simp[k])\cong c^h(\Simp[0]\boxtimes c\Sd^2\Simp[k])\cong [0]\boxtimes c\Sd^2\Simp[k] \cong \bV c\Sd^2\Simp[k]. \]
    As these isomorphisms are natural in $k$, we get a natural isomorphism $c^h(c\Sd^2)_*\cst\cong \bV c\Sd^2$, as desired.    
\end{proof}

\begin{prop} \label{prop:CatvsDblCat}
    The adjunction 
    \begin{tz}
\node[](A) {$\Cat_\Thom$};
\node[right of=A,xshift=1.8cm](B) {$\DblCat_{\Gpd}$};
\draw[->] ($(A.east)+(0,5pt)$) to node[above,la]{$\bV$} ($(B.west)+(0,5pt)$);
\draw[->] ($(B.west)-(0,5pt)$) to node[below,la]{$\bfV$} ($(A.east)-(0,5pt)$);
\node[la] at ($(A.east)!0.5!(B.west)$) {$\bot$};
\end{tz}
    is a Quillen equivalence. 
\end{prop}

\begin{proof}
    By \cref{cstvsV}, we have a commutative square of adjunctions
    \begin{tz}
    \node[](1) {$\sSet_\Kan$}; 
    \node[below of=1,yshift=-.5cm](2) {$\Cat_\Thom$}; 
    \node[right of=1,xshift=1.8cm](3) {$\sSet^{\Simpop}_\Gpd$}; 
    \node[below of=3,yshift=-.5cm](4) {$\DblCat_\Gpd$};
\punctuation{4}{,};

    \draw[->] ($(1.east)+(0,5pt)$) to node[above,la]{$\cst$} ($(3.west)+(0,5pt)$);
\draw[->] ($(3.west)-(0,5pt)$) to node[below,la]{$(-)_0$} ($(1.east)-(0,5pt)$);
\node[la] at ($(1.east)!0.5!(3.west)$) {$\bot$};
\draw[->] ($(2.east)+(0,5pt)$) to node[above,la]{$\bV$} ($(4.west)+(0,5pt)$);
\draw[->] ($(4.west)-(0,5pt)$) to node[below,la]{$\bfV$} ($(2.east)-(0,5pt)$);
\node[la] at ($(2.east)!0.5!(4.west)$) {$\bot$};

\draw[->] ($(3.south)-(5pt,0)$) to node[left,la]{$\ch(c\Sd^2)_*$} ($(4.north)-(5pt,0)$); 
\draw[->] ($(4.north)+(5pt,0)$) to node[right,la]{$(\Ex^2 N)_*\Nh$} ($(3.south)+(5pt,0)$); 
\node[la] at ($(3.south)!0.5!(4.north)$) {\rotatebox{90}{$\bot$}};
\draw[->] ($(1.south)-(5pt,0)$) to node[left,la]{$c\Sd^2$} ($(2.north)-(5pt,0)$); 
\draw[->] ($(2.north)+(5pt,0)$) to node[right,la]{$\Ex^2 N$} ($(1.south)+(5pt,0)$); 
\node[la] at ($(1.south)!0.5!(2.north)$) {\rotatebox{90}{$\bot$}};
\end{tz}
    where the left, right, and top adjunctions are Quillen equivalences by \cref{QEThomason,thm:gpdsquare,modelofgpd}. Using that $\Cat_\Thom$ and $\DblCat_\Gpd$ are right-induced along the vertical adjunctions from $\sSet_\Kan$ and $\sSet^{\Simpop}_\Gpd$, respectively, we see that the bottom adjunction is a Quillen pair, which is then also a Quillen equivalence by $2$-out-of-$3$.
\end{proof}

\subsection{Comparison with Fiore--Paoli's model structure on \texorpdfstring{$\DblCat$}{DblCat}} \label{sec:GpdvsThom}

Fiore--Paoli have considered yet another model structure on double categories in \cite{FP} giving a model of $\infty$-groupoids. We now recall their result.

\begin{rem}
    The diagonal functor $\diag\colon \Simp\to \Simp\times \Simp$ induces by left Kan extension and pre-composition an adjunction
    \begin{tz}
\node[](B) {$\sSet$};
\node[right of=B,xshift=1.2cm](C) {$\sSet^{\Simpop}$};
\punctuation{C}{,};
\draw[->] ($(B.east)+(0,5pt)$) to node[above,la]{$\diag_!$} ($(C.west)+(0,5pt)$);
\draw[->] ($(C.west)-(0,5pt)$) to node[below,la]{$\diag^*$} ($(B.east)-(0,5pt)$);
\node[la] at ($(B.east)!0.5!(C.west)$) {$\bot$};
\end{tz}
\end{rem}

The following result appears as the case $n=2$ of \cite[Theorem 8.2]{FP}.

\begin{theorem}
The right-induced model structure on $\DblCat$ from the Kan--Quillen model structure $\sSet_\Kan$ along the composite of adjunctions
\begin{tz}
\node[](A) {$\sSet_\Kan$};
\node[right of=A,xshift=1.2cm](B) {$\sSet$};
\node[right of=B,xshift=1.2cm](C) {$\sSet^{\Simpop}$};
\node[right of=C,xshift=1.4cm](D) {$\DblCat$};
\draw[->] ($(A.east)+(0,5pt)$) to node[above,la]{$\Sd^2$} ($(B.west)+(0,5pt)$);
\draw[->] ($(B.west)-(0,5pt)$) to node[below,la]{$\Ex^2$} ($(A.east)-(0,5pt)$);
\node[la] at ($(A.east)!0.5!(B.west)$) {$\bot$};
\draw[->] ($(C.east)+(0,5pt)$) to node[above,la]{$\bC$} ($(D.west)+(0,5pt)$);
\draw[->] ($(D.west)-(0,5pt)$) to node[below,la]{$\bN$} ($(C.east)-(0,5pt)$);
\node[la] at ($(C.east)!0.5!(D.west)$) {$\bot$};
\draw[->] ($(B.east)+(0,5pt)$) to node[above,la]{$\diag_!$} ($(C.west)+(0,5pt)$);
\draw[->] ($(C.west)-(0,5pt)$) to node[below,la]{$\diag^*$} ($(B.east)-(0,5pt)$);
\node[la] at ($(B.east)!0.5!(C.west)$) {$\bot$};
\end{tz}
exists. It is called the \emph{Thomason model structure} and we denote it by $\DblCat_\Thom$. 

Moreover, the model structure $\DblCat_\Thom$ is cofibrantly generated with generating sets of cofibrations and trivial cofibrations given by 
\[ \{ \bC\diag_!\Sd^2(\partial\Simp[k]\to \Simp[k])\mid k\geq 0\} \quad \text{and}\quad \{ \bC\diag_!\Sd^2(\Lambda^t[k]\to \Simp[k])\mid k\geq 1, 0\leq t\leq k\}, \]
respectively.
\end{theorem}

The following characterization of the weak equivalences in $\DblCat_\Thom$ can be found in \cite[Proposition 8.1]{FP}, or can be deduced from \cref{natbeta}.

\begin{prop} \label{weinDblCatThom}
    A double functor $F$ is a weak equivalence in $\DblCat_\Thom$ if and only if the diagonal of its double nerve $\diag^*\bN F$ is a weak equivalence in $\sSet_\Kan$. 
\end{prop}

The Thomason model structure on $\DblCat$ is proven to be a model of $\infty$-groupoids in the case $n=2$ of \cite[Proposition 9.27]{FP}.

\begin{prop}
    The Quillen pair
\begin{tz}
\node[](A) {$\sSet_\Kan$};
\node[right of=A,xshift=2.2cm](B) {$\DblCat_\Thom$};
\draw[->] ($(A.east)+(0,5pt)$) to node[above,la]{$\bC\diag_!\Sd^2$} ($(B.west)+(0,5pt)$);
\draw[->] ($(B.west)-(0,5pt)$) to node[below,la]{$\Ex^2\diag^*\bN$} ($(A.east)-(0,5pt)$);
\node[la] at ($(A.east)!0.5!(B.west)$) {$\bot$};
\end{tz}
    is a Quillen equivalence.
\end{prop}

\begin{rem}
   One might hope to establish a direct Quillen equivalence between $\Cat_\Thom$ and $\DblCat_\Thom$, similar to the one obtained for $\DblCat_\Gpd$ in \cref{prop:CatvsDblCat}. However, this seems to be obstructed by the fact that the functor $\diag^*\colon \sSet^{\Simpop} \to \sSet$ does \emph{not} restrict to a functor $\diag^*\colon \DblCat \to \Cat$.
\end{rem}

In particular, the model structures $\DblCat_{\Gpd}$ and $\DblCat_{\Thom}$ represent the same homotopy theory, namely that of $\infty$-groupoids. Furthermore, they have the same weak equivalences, as we now show. 

\begin{rem} \label{rem:diagofwe}
    By \cite[Theorem 2.11]{Rasekh}, a map $f$ is a weak equivalence in the localization of the injective model structure on $(\sSet_\Kan)^{\Simpop}$ at the set $\Gpd$ if and only if its diagonal $\diag^* f$ is a weak equivalence in $\sSet_\Kan$. Since the localization $\sSet^{\Simpop}_\Gpd$ has the same weak equivalences, we deduce that $f$ is a weak equivalence in $\sSet^{\Simpop}_\Gpd$ if and only if $\diag^* f$ is a weak equivalence in $\sSet_\Kan$.
\end{rem}

\begin{prop}
    A double functor $F$ is a weak equivalence in $\DblCat_\Gpd$ if and only if it is a weak equivalence in $\DblCat_\Thom$.
\end{prop}

\begin{proof}
    By \cref{cor:charweGpd}, the double functor $F$ is a weak equivalence in $\DblCat_\Gpd$ if and only if its double nerve $\bN F$ is a weak equivalence in $\sSet^{\Simpop}_\Gpd$. By \cref{rem:diagofwe}, this is the case if and only if its diagonal $\diag^*\bN F$ is a weak equivalence in $\sSet_\Kan$. By \cref{weinDblCatThom}, this holds if and only if $F$ is a weak equivalence in $\DblCat_\Thom$.
\end{proof}

\begin{cor}
    The identity functors in both directions induce equivalences of $\infty$-categories 
    \[ (\DblCat_\Gpd)_\infty\simeq (\DblCat_\Thom)_\infty. \]
\end{cor}

However, while the identity functor induces an equivalence of underlying $\infty$-categories, it fails to provide a Quillen equivalence, as shown by the next two results.

\begin{rem} \label{rem:cSd1}
    Recall that $\Sd^2\Simp[1]$ can be computed as the following pushout in $\sSet$.
    \begin{tz}
        \node[](1) {$\Simp[0]$}; 
        \node[below of=1](2) {$\Simp[1]$}; 
        \node[right of=1,xshift=1.6cm](3) {$\Simp[1]$}; 
        \node[below of=3](4) {$\Simp[1]\amalg_{\Simp[0]} \Simp[1]$}; 
        \pushout{4};

        \draw[->](1) to node[left,la]{$d^0$} (2); 
        \draw[->](1) to node[above,la]{$d^0$} (3); 
        \draw[->](3) to node[right,la]{$\iota_1$} (4); 
          \draw[->](2) to node[below,la]{$\iota_2$} (4); 
    \end{tz}
\end{rem}

\begin{prop}
    The identity functor $\id\colon \DblCat_\Gpd\to \DblCat_\Thom$ is not left Quillen.
\end{prop}

\begin{proof}
     We show that the generating cofibration $[0]\boxtimes (c\Sd^2(\partial \Simp[1]\to \Simp[1])$ in $\DblCat_\Gpd$ is not a cofibration in $\DblCat_\Thom$. Using \cref{rem:cSd1}, this cofibration is given by the double functor 
     \[ \id_{[0]}\boxtimes (\iota_1 d^1 + \iota_2 d^1)\colon [0]\boxtimes ([0]\amalg [0])\to [0]\boxtimes ([1]\amalg_{[0]} [1]). \]
     Next, note that the following double functor is a trivial fibration in $\DblCat_\Thom$.
     \[ !\boxtimes \id_{[1]}\colon [1]\boxtimes [1]\to [0]\boxtimes [1] \]
    Indeed, after applying the functor $\diag^* \bN\colon \DblCat\to \sSet$, we obtain the trivial fibration in $\sSet_\Kan$
    \[ !\times \id_{\Simp[1]}\colon \Simp[1]\times \Simp[1]\to \Simp[1], \]
    which remains a trivial fibration after applying the functor $\Ex^2$ by \cite[Corollary 2.1.27]{Cisinski} and \cref{natbeta}. However, there is no lift in the following commutative diagram in~$\DblCat$.
    \begin{tz}
        \node[](1) {$[0]\boxtimes ([0]\amalg [0])$}; 
        \node[below of=1](2) {$[0]\boxtimes ([1]\amalg_{[0]} [1])$}; 
        \node[right of=1,xshift=3.3cm](3) {$[1]\boxtimes [1]$}; 
        \node[below of=3](4) {$[0]\boxtimes [1]$}; 

        \draw[->](1) to node[left,la]{$\id_{[0]}\boxtimes (\iota_1 d^1 + \iota_2 d^1)$} (2); 
        \draw[->](1) to node[above,la]{$(d^1\boxtimes d^1) + (d^0\boxtimes d^0)$} (3); 
        \draw[->](3) to node[right,la]{$!\boxtimes \id_{[1]}$} (4); 
          \draw[->](2) to node[below,la]{$\id_{[0]}\boxtimes (! +\id_{[1]})$} (4); 
    \end{tz}
    This shows that $\id_{[0]}\boxtimes (\iota_1 d^1 + \iota_2 d^1)$ is not a cofibration in $\MSdiag$.
\end{proof}

\begin{prop}
    The identity functor $\id\colon \DblCat_\Thom\to \DblCat_\Gpd$ is not left Quillen.
\end{prop}

\begin{proof}
    We show that the generating cofibration $\bC \diag_! \Sd^2(\partial \Simp[1]\to \Simp[1])$ in $\DblCat_\Thom$ is not a cofibration in $\DblCat_\Gpd$. Using \cref{rem:cSd1}, this cofibration is given by the double functor
    \[
    \iota_1(d^1\boxtimes d^1)+\iota_2(d^0\boxtimes d^0)\colon ([0]\boxtimes [0])\amalg ([0]\boxtimes [0])\to ([1]\boxtimes [1])\amalg_{[0]\boxtimes [0]} ([1]\boxtimes [1]). \]
    Next, consider the following pushout diagram in $\DblCat$.
    \begin{tz}
        \node[](1) {$[0]\boxtimes [0]$}; 
        \node[below of=1](2) {$[0]\boxtimes [1]$}; 
        \node[right of=1,xshift=1.5cm](3) {$[1]\boxtimes [0]$}; 
        \node[below of=3](4) {$\bP$}; 
        \pushout{4};
        \node[below of=4,xshift=2cm](5) {$[1]\boxtimes [0]$}; 

        \draw[->](1) to node[left,la]{$\id_{[0]}\boxtimes d^1$} (2); 
        \draw[->](1) to node[above,la]{$d^1\boxtimes \id_{[0]}$} (3); 
        \draw[->](3) to node[right,la]{$\iota_{1,0}$} (4); 
          \draw[->](2) to node[below,la]{$\iota_{0,1}$} (4);
          \draw[->,dashed](4) to (5); 
          \draw[->,bend right=15](2) to node[below,la,yshift=-3pt]{$\id_{[0]}\boxtimes !$} (5);
          \draw[->,bend left=20](3) to node[right,la]{$\id_{[1]\boxtimes [0]}$} (5);
    \end{tz}
    We show that the double functor $\bP\to [1]\boxtimes [0]$ is a trivial fibration in $\DblCat_\Gpd$. For this, note that the double nerve $\bN\colon \DblCat\to \sSet^{\Simpop}$ preserves the pushout $\bP$. Hence, the map obtained after applying the functor $\bN\colon \DblCat\to \sSet^{\Simpop}$ is given at level $n\geq 0$ by the trivial fibration in $\sSet_\Kan$
    \[ \textstyle (\coprod_{n+1} \id_{\Simp[0]}) +!\colon (\coprod_{n+1} \Simp[0])\amalg \Simp[1]\to \coprod_{n+2} \Simp[0], \]
    which remains a trivial fibration after applying the functor $\Ex^2$ by \cite[Corollary 2.1.27]{Cisinski} and \cref{natbeta}. Hence the map $(\Ex^2)_* \bN(\bP\to [1]\boxtimes [0])$ is a trivial fibration in $\sSet^{\Simpop}_\proj$, as desired. However, there is no lift in the following commutative diagram in~$\DblCat$.
    \begin{tz}
        \node[](1) {$([0]\boxtimes [0])\amalg ([0]\boxtimes [0])$}; 
        \node[below of=1](2) {$([1]\boxtimes [1])\amalg_{[0]\boxtimes[0]} ([1]\boxtimes [1])$}; 
        \node[right of=1,xshift=5cm](3) {$\bP$}; 
        \node[below of=3](4) {$[1]\boxtimes [0]$}; 

        \draw[->](1) to node[left,la]{$\iota_1(d^1\boxtimes d^1)+\iota_2(d^0\boxtimes d^0)$} (2); 
        \draw[->](1) to node[above,la]{$\iota_{0,1}(\id_{[0]} \boxtimes d^0) + \iota_{1,0}(d^0 \boxtimes \id_{[0]})$} (3); 
        \draw[->](3) to (4); 
          \draw[->](2) to node[below,la]{$(\id_{[1]}\boxtimes !)+ (!\boxtimes !)$} (4); 
    \end{tz}
    This shows that $\iota_1(d^1\boxtimes d^1)+\iota_2(d^0\boxtimes d^0)$ is not a cofibration in $\DblCat_\Gpd$.
\end{proof}

\section{Homotopy colimits}

In this section, we give an explicit formula for the homotopy colimits in the model structures on $\DblCat$ considered here. To do so, we first recall in \cref{sec:hocolimcat} that the Grothendieck construction serves as a model for the homotopy colimit of diagrams valued in the Thomason model structure on $\Cat$. Then, in \cref{sec:hocolimdblcat}, we show that applying the Grothendieck construction levelwise yields a model for the homotopy colimit of diagrams valued in any of the model structures on $\DblCat$ constructed in this paper. Finally, in \cref{sec:application}, we compute the homotopy colimits in $\DblCat$ used to describe the sets $\Seg$ and~$\CSS$. 

\subsection{Homotopy colimits in Thomason's model structure on \texorpdfstring{$\Cat$}{Cat}} \label{sec:hocolimcat}

Let us fix a small category $\sJ$. We start by recalling the Grothendieck construction of a $\Cat$-valued functor.

\begin{defn} \label{def:GC}
    Given a functor $F \colon \sJ\to \Cat$, its \textbf{Grothendieck construction} is the category $\int_\sJ F$ such that:
\begin{itemize}[leftmargin=0.6cm]
    \item an object is a pair $(j,x)$ of an object of $j\in \sJ$ and an object of $x\in Fj$,
    \item a morphism $(j,x) \to (j',x')$ is a pair $(s,u)$ of a morphism $s \colon j\to j'$ in $\sJ$ and a morphism $u \colon Fs(x) \to x'$ in $Fj'$, with 
    \begin{itemize}
        \item the composite of $(s,u) \colon (j,x) \to (j',x')$ and $(s',u') \colon (j',x') \to (j'',x'')$ given by $(s'\circ s,u'\circ Fs'(u))$, 
        \item identity at $(j,x)$ is given by $(\id_j,\id_{x})$.
    \end{itemize}
\end{itemize}
This construction extends to a functor $\int_\sJ\colon \Cat^\sJ\to \Cat$. 
\end{defn}

Recall now the following theorem due to Thomason, whose original statement appears as \cite[Theorem 1.2]{Thomasoncolim}.

\begin{theorem} \label{GChocolim}
    The Grothendieck construction
    \[ \textstyle \int_\sJ \colon (\Cat_\Thom)^\sJ \to \Cat_\Thom\]
    is a model for the homotopy colimit functor in the Thomason model structure $\Cat_\Thom$.
\end{theorem}

\begin{proof}
    First, recall that $\int_J$ preserves weak equivalences by \cite[Proposition~2.3.1]{Maltsiniotis}. By \cite[Proposition 3.1.2]{Maltsiniotis}, the functor $\int_{\sJ}\colon \Cat^{\sJ}\to \Cat$ has a right adjoint given by the functor
    \[
   \cH_{\sJ} \colon \Cat \to \Cat^{\sJ}, \quad
        \sC \mapsto \sC^{\sJ_{-/}}.
    \]
    Since for every object $j\in \sJ$, the slice category $\sJ_{j/}$ has an initial object, the unique functor $\sJ_{j/} \to [0]$ admits a left adjoint. As $\sC^{(-)}\colon \Cat^{\op}\to \Cat$ is a $2$-functor, the induced functor
    \[
    \sC\cong \sC^{[0]} \longrightarrow \sC^{\sJ_{j/}}=\cH_\sJ(\sC)(j)
    \]
    admits a right adjoint, and so it is a weak equivalence in $\Cat_\Thom$; see \cite[Chapter 1, Corollary~1]{QuillenKTheory} or \cite[Proposition 1.1.9]{Maltsiniotis}. In other words, we have a canonical natural transformation of functors $\Cat\to \Cat^\sJ$
    \[
    \cst\Longrightarrow  \cH_{\sJ} 
    \]
    which is levelwise a weak equivalence in $(\Cat_\Thom)^{\sJ}_\proj$. In particular, since $\cst$ preserves all weak equivalences, this proves that $\cH_{\sJ}$ also does. By adjunction, we also have a natural transformation of functors $\Cat^{\sJ} \to \Cat$
    \[
    \textstyle \int_{\sJ} \Longrightarrow \colim_{\sJ}.
    \]
    Now, let $Q \colon \Cat^{\sJ} \to \Cat^{\sJ}$ be a chosen cofibrant replacement functor for the model structure $(\Cat_\Thom)^{\sJ}_\proj$, so that $\hocolim_{\sJ}=\colim_{\sJ}\circ Q$. In particular, by precomposing with $Q$, we obtain adjunctions
    \[
    \textstyle \int_{\sJ}\circ Q \dashv \cH_{\sJ}\circ Q \quad \text{and} \quad \hocolim_{\sJ} \dashv \cst \circ Q 
    \]
    and natural transformations
    \[
    \textstyle \int_{\sJ}\circ Q \Longrightarrow \hocolim_{\sJ} \quad \text{and} \quad \cst \circ Q \Longrightarrow \cH_{\sJ}\circ Q.
    \]
    All the functors involved preserve weak equivalences, so we get induced adjunctions and natural transformations at the level of underlying $\infty$-categories. As $\cst \circ Q \Rightarrow \cH_{\sJ}\circ Q$ is levelwise a weak equivalence in $(\Cat_\Thom)^{\sJ}_\proj$, the induced natural transformation at the level of underlying $\infty$-categories is a natural isomorphism, and thus so is the natural transformation induced by its mate $\int_{\sJ}\circ Q \Rightarrow \hocolim_{\sJ}$. This proves that the latter is levelwise a weak equivalence in $\Cat_\Thom$ by \cref{rem:liftQE}. 

    Hence we have a zigzag of levelwise weak equivalences in $\Cat_\Thom$
    \[
    \textstyle \int_J \Longleftarrow \int_J \circ Q \Longrightarrow \hocolim_{\sJ},
    \]
    where the left-hand natural transformation is induced by $Q \Rightarrow \id$. In other words, we have that $\int_\sJ\simeq \hocolim_\sJ$, as desired.
\end{proof}

The upshot of the previous result is that the Grothendieck construction has an easy explicit description, making it much more convenient to use than, for instance, $\colim_\sJ \circ Q$, where $Q\colon \Cat \to \Cat$ is a cofibrant replacement functor for $\Cat_\Thom$.

\subsection{Homotopy colimits in the model structures on \texorpdfstring{$\DblCat$}{DblCat}} \label{sec:hocolimdblcat}

We now aim to show that we also have an easy explicit formula for homotopy colimits in the different model structures on $\DblCat$ considered in the paper. To achieve this, we first study the case of the corresponding model structures on $\Cat^{\Simpop}$.

\begin{rem}
    The Grothendieck construction functor $\int_\sJ\colon \Cat^\sJ\to \Cat$ induces by post-compo\-sition a functor 
    \[ \textstyle (\int_\sJ)_*\colon \Cat^{\Simpop\times \sJ}\cong (\Cat^\sJ)^{\Simpop}\to \Cat^{\Simpop}. \]
    It sends a functor $F \colon \Simpop\times \sJ \to \Cat$ to the functor 
    \[ \textstyle \Simpop\to \Cat, \quad [n]\mapsto \int_\sJ F([n],-). \]
\end{rem}

\begin{prop} \label{GCinproj}
    Let $\sJ$ be a small category. The functor
    \[
    \textstyle (\int_\sJ)_* \colon (\Cat_{\proj}^{\Simpop})^\sJ \to \Cat_{\proj}^{\Simpop}
    \]
    is a model for the homotopy colimit functor.
\end{prop}

\begin{proof}
    This is obtained by combining \cref{GChocolim} with \cref{hocoliminproj} in the case where $\sK=\Simpop$ and $\sM=\Cat_\Thom$.
\end{proof}

\begin{prop} \label{GCinloc}
    Let $\sJ$ be a small category. The functors
    \[
    \textstyle (\int_\sJ)_* \colon (\Cat_{\Seg}^{\Simpop})^\sJ \to \Cat_{\Seg}^{\Simpop}, \quad (\int_\sJ)_* \colon (\Cat_{\CSS}^{\Simpop})^\sJ \to \Cat_{\CSS}^{\Simpop}, \quad (\int_\sJ)_* \colon (\Cat_{\Gpd}^{\Simpop})^\sJ \to \Cat_{\Gpd}^{\Simpop}
    \]
    are models for the homotopy colimit functor.
\end{prop}

\begin{proof}
    This is obtained by combining \cref{GCinproj} with \cref{hocoliminloc} in the case where $\sM=\Cat^{\Simpop}_\proj$ and $S=\Seg,\CSS,\Gpd$.
\end{proof}

Building on the above results, we now show that the functor $(\int_\sJ)_*\colon (\Cat^{\Simpop})^\sJ\to \Cat^{\Simpop}$ restricts appropriately to a functor $\DblCat^\sJ\to \DblCat$, and provides a model for the homotopy colimit in the corresponding model structures on $\DblCat$.

\begin{lemma}\label{lemma:intpullback}
    The Grothendieck construction functor $\int_\sJ \colon \Cat^\sJ\to \Cat$ preserves pullbacks.
\end{lemma}

\begin{proof}
    Since, for every functor $F\colon \sJ\to \Cat$, there is a canonical projection $\int_\sJ F\to \sJ$, the Grothendieck construction defines a functor $\int_\sJ\colon \Cat^\sJ\to \Cat_{/\sJ}$. By \cite[Proposition 3.1.2]{Maltsiniotis}, this functor is a right adjoint. As a consequence, the functor $\int_\sJ\colon \Cat^\sJ\to \Cat$ preserves all connected limits, and in particular pullbacks.
\end{proof}

\begin{prop}
    The functor $(\int_\sJ)_*\colon \Cat^{\Simpop\times \sJ}\cong (\Cat^\sJ)^{\Simpop}\to \Cat^\sJ$ restricts to a functor $(\int_\sJ)_* \colon \DblCat^\sJ\to \DblCat$ making the following square commute.
    \begin{tz}
        \node[](1) {$\DblCat^\sJ$}; 
        \node[below of=1](2) {$(\Cat^{\Simpop})^\sJ$}; 
        \node[right of=1,xshift=1.6cm](3) {$\DblCat$}; 
        \node[below of=3](4) {$\Cat^{\Simpop}$}; 

        \draw[->](1) to node[left,la]{$(N_h)_*$} (2); 
        \draw[->](1) to node[above,la]{$(\int_\sJ)_*$} (3); 
        \draw[->](3) to node[right,la]{$N_h$} (4); 
          \draw[->](2) to node[below,la]{$(\int_\sJ)_*$} (4); 
    \end{tz}
\end{prop}

\begin{proof}
    The horizontal nerve $\Nh\colon \DblCat\to \Cat^{\Simpop}$ is fully faithful with essential image given by the simplicial objects in $\Cat$ satisfying the (strict) Segal conditions. Since these conditions are expressed in terms of iterated pullbacks, we conclude by \cref{lemma:intpullback}.
\end{proof}

In particular, one can extract the following explicit description of the double categories that arise as images of the functor $(\int_\sJ)_*$.

\begin{rem} \label{explicitGCdblcat}
    Given a functor $F \colon \sJ\to \DblCat$, then the double category $(\int_\sJ)_* F$ is such that:
\begin{itemize}[leftmargin=0.6cm]
    \item an object is a pair $(j,x)$ of an object of $j\in \sJ$ and an object of $x\in Fj$,
    \item a vertical morphism $(j,x) \varrow (j',x')$ is a pair $(s,u)$ of a morphism $s \colon j\to j'$ in $\sJ$ and a vertical morphism $u \colon Fs(x) \varrow x'$ in $Fj'$, with 
    \begin{itemize}
        \item compositions and identities defined as in \cref{def:GC},
    \end{itemize}
    \item a horizontal morphism $(j,x)\to (j,y)$ is a horizontal morphism $f\colon x\to y$ in $Fj$, with
    \begin{itemize}
        \item compositions and identities given by those of $Fj$;
    \end{itemize}
    there is no horizontal morphism between $(j,x)$ and $(k,y)$ for $j\neq k$, 
    \item a square as below left is a square in $Fj'$ as below right,
    \begin{tz}
        \node[](1) {$(j,x)$}; 
        \node[below of=1](2) {$(j',x')$}; 
        \node[right of=1,xshift=.9cm](3) {$(j,y)$}; 
        \node[below of=3](4) {$(j',y')$}; 
        
        \draw[->,pro] (1) to node[left,la]{$(s,u)$} (2); 
        \draw[->] (1) to node[above,la]{$f$} (3);
        \draw[->] (2) to node[below,la]{$f'$} (4); 
        \draw[->,pro] (3) to node[right,la]{$(s,v)$} (4);
        
        \node at ($(1)!0.5!(4)$) {\rotatebox{270}{$\Rightarrow$}};

        \node[right of=3,xshift=2cm](1) {$Fs(x)$}; 
        \node[below of=1](2) {$x'$}; 
        \node[right of=1,xshift=.9cm](3) {$Fs(y)$}; 
        \node[below of=3](4) {$y'$}; 
        
       \draw[->,pro] (1) to node[left,la]{$u$} (2); 
        \draw[->] (1) to node[above,la]{$Fs(f)$} (3);
        \draw[->] (2) to node[below,la]{$f'$} (4); 
        \draw[->,pro] (3) to node[right,la]{$v$} (4);
        
        \node at ($(1)!0.5!(4)$) {\rotatebox{270}{$\Rightarrow$}};
        \end{tz}
        with 
        \begin{itemize}
        \item vertical compositions and identities defined as in \cref{def:GC},
        \item horizontal compositions and identities given by those of $Fj'$;
    \end{itemize}
        there is no square between $(s,u)$ and $(t,v)$ for $s\neq t$.
    \end{itemize}
\end{rem}

We conclude with the desired result. 

\begin{theorem}\label{thm:homotopycolim}
    Let $\sJ$ be a small category. The functors
    \[
    \textstyle (\int_\sJ)_* \colon (\DblCat_{\proj})^\sJ \to \DblCat_{\proj}, \quad (\int_\sJ)_* \colon (\DblCat_{\Seg})^\sJ \to \DblCat_{\Seg}, \]\[ \textstyle(\int_\sJ)_* \colon (\DblCat_{\CSS})^\sJ \to \DblCat_{\CSS}, \quad (\int_\sJ)_* \colon (\DblCat_{\Gpd})^\sJ \to \DblCat_{\Gpd}
    \]
    are models for the homotopy colimit functor.
\end{theorem}

\begin{proof}
    This follows directly from the fact that in the commutative diagram 
    \begin{tz}
        \node[](1) {$(\DblCat_\mathrm{xxx})^\sJ_\proj$}; 
        \node[below of=1](2) {$(\Cat^{\Simpop}_\mathrm{xxx})^\sJ_\proj$}; 
        \node[right of=1,xshift=1.8cm](3) {$\DblCat_\mathrm{xxx}$}; 
        \node[below of=3](4) {$\Cat^{\Simpop}_\mathrm{xxx}$}; 
\punctuation{4}{,};

        \draw[->](1) to node[left,la]{$(N_h)_*$} (2); 
        \draw[->](1) to node[above,la]{$(\int_\sJ)_*$} (3); 
        \draw[->](3) to node[right,la]{$N_h$} (4); 
          \draw[->](2) to node[below,la]{$(\int_\sJ)_*$} (4); 
    \end{tz}
    where $\mathrm{xxx}=\proj,\,\Seg,\,\CSS,$ or $\Gpd$, the vertical functors are Quillen equivalences by \cref{thm:allQE,squareSeg,squareCSS,thm:gpdsquare} and \cite[Theorem 11.6.5]{Hirschhorn}. Moreover, by \cref{GCinproj,GCinloc}, the bottom functor is a model for the homotopy colimit. Hence, the top functor is also a model for the homotopy colimit functor, as desired. 
\end{proof}

\subsection{Application} \label{sec:application}

As an application, we compute the homotopy colimits in $\DblCat_\proj$ present in the localization sets of \cref{rem:explicitsetSeg,rem:explicitsetCSS}.

We start with the spine inclusion $\bH[1]\amalg_{ \bH[0]}^h\ldots \amalg_{\bH[0]}^h \bH[1]\to \bH[n]$ in $\DblCat_\proj$. The source of this double functor is given by the homotopy colimit in $\DblCat_\proj$ of the diagram
\begin{tz}
    \node[](1) {$\bH[1]$}; 
    \node[above right of=1,xshift=.5cm](2) {$\bH[0]$};
    \node[below right of=2,xshift=.5cm](3) {$\bH[1]$}; 
    \node[above right of=3,xshift=.5cm](4){$\bH[0]$}; 
    \node[below right of=4,white,xshift=.5cm](5) {$\Simp[1]$}; 
    \node[below right of=4,xshift=.5cm] {$\ldots$}; 
    \node[above right of=5,xshift=.5cm](6){$\bH[0]$};
    \node[below right of=6,xshift=.5cm](7) {$\bH[1]$}; 

    \draw[->](2) to node[above,la]{$d^0$} (1);
    \draw[->](2) to node[above,la,xshift=3pt]{$d^1$} (3);
    \draw[->](4) to node[above,la]{$d^0$} (3);
    \draw[->](4) to node[above,la,xshift=3pt]{$d^1$} (5);
    \draw[->](6) to node[above,la]{$d^0$} (5);
    \draw[->](6) to node[above,la,xshift=3pt]{$d^1$} (7);
\end{tz}
By \cref{thm:homotopycolim} applied to the category $\sJ=\{\bullet\leftarrow \bullet \to \bullet\leftarrow \ldots \leftarrow\bullet \to\bullet\}$, we can compute this homotopy colimit by applying the functor $(\int_\sJ)_*\colon \DblCat^{\sJ}\to \DblCat$ to the functor $\sJ\to \DblCat$ representing the above diagram. Using \cref{explicitGCdblcat}, we get that it is given by the double category of the form
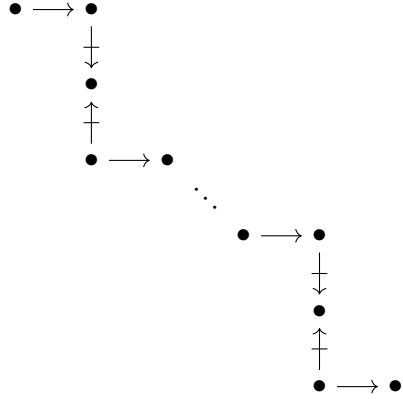
\begin{tz}[node distance=1cm]
    \node[](1) {$\bullet$}; 
    \node[right of=1](2) {$\bullet$};
    \node[below of=2](3) {$\bullet$};
    \node[below of=3](4) {$\bullet$};
    \node[right of=4](5) {$\bullet$};
    \node[below of=5,xshift=1cm](9) {$\bullet$};
    \node[right of=9](9') {$\bullet$};
    \node[below of=9'](10) {$\bullet$};
    \node[below of=10](11) {$\bullet$};
    \node[right of=11](12) {$\bullet$};
    \draw[->](1) to (2);
    \draw[->,pro](2) to (3);
    \draw[->,pro](4) to (3);
    \draw[->](4) to (5);
    \node at ($(5)!0.5!(9)$) {\rotatebox{135}{$\ldots$}};
    \draw[->](9) to (9');
    \draw[->,pro](9') to (10);
    \draw[->,pro](11) to (10);
    \draw[->](11) to (12);
\end{tz}
Moreover, the spine inclusion $\bH[1]\amalg_{ \bH[0]}^h\ldots \amalg_{\bH[0]}^h \bH[1]\to \bH[n]$ is given by the canonical projection double functor sending all vertical morphisms to identities.

We now turn to the completeness map $\bH[0]\amalg^h_{\bH[1]}\amalg^h_{\bH[3]} \amalg^h_{\bH[1]} \bH[0]\to \bH[0]$ in $\DblCat_\proj$. The source of this double functor is given by the homotopy colimit in $\DblCat_\proj$ of the diagram
    \begin{tz}
    \node[](1) {$\bH[0]$}; 
    \node[above right of=1,xshift=.5cm](2) {$\bH[1]$};
    \node[below right of=2,xshift=.5cm](3) {$\bH[3]$}; 
    \node[above right of=3,xshift=.5cm](4){$\bH[1]$}; 
    \node[below right of=4,xshift=.5cm](5) {$\bH[0]$};  

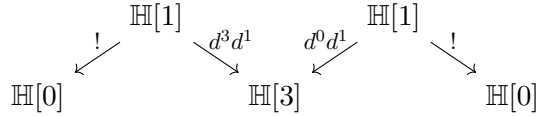
    \draw[->](2) to node[above,la]{$!$} (1);
    \draw[->](2) to node[above,la,xshift=7pt]{$d^3d^1$} (3);
    \draw[->](4) to node[above,la,xshift=-4pt]{$d^0d^1$} (3);
    \draw[->](4) to node[above,la]{$!$} (5);
\end{tz}
Reasoning as before, we get that it is given by the double category freely generated by the following data
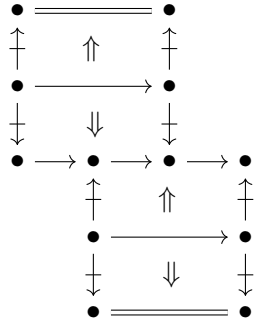
\begin{tz}[node distance=1cm]
    \node[](1) {$\bullet$}; 
    \node[below of=1](2) {$\bullet$};
    \node[below of=2](3) {$\bullet$}; 
    \node[right of=3](4) {$\bullet$}; 
    \node[right of=4](5) {$\bullet$}; 
    \node[above of=5](6) {$\bullet$}; 
    \node[above of=6](7) {$\bullet$}; 
    \draw[d](1) to (7);
    \draw[->](2) to (6); 
    \draw[->](3) to (4);
    \draw[->](4) to (5);
    \draw[->,pro](2) to (1);
    \draw[->,pro](2) to (3);
    \draw[->,pro](6) to (5);
    \draw[->,pro](6) to (7);

    \node[right of=5](8) {$\bullet$}; 
    \node[below of=4](9) {$\bullet$}; 
    \node[below of=9](10) {$\bullet$}; 
    \node[below of=8](11) {$\bullet$}; 
    \node[below of=11](12) {$\bullet$};
    \draw[->](5) to (8);
    \draw[->](9) to (11);
    \draw[d](10) to (12);
    \draw[->,pro](9) to (4);
    \draw[->,pro](9) to (10);
    \draw[->,pro](11) to (8);
    \draw[->,pro](11) to (12);

    \node[] at ($(1)!0.5!(6)$) {\rotatebox{90}{$\Rightarrow$}};
    \node[] at ($(2)!0.5!(5)$) {\rotatebox{270}{$\Rightarrow$}};
    \node[] at ($(4)!0.5!(11)$) {\rotatebox{90}{$\Rightarrow$}};
    \node[] at ($(9)!0.5!(12)$) {\rotatebox{270}{$\Rightarrow$}};
\end{tz}
Finally, the completeness map $\bH[0]\amalg^h_{\bH[1]}\amalg^h_{\bH[3]} \amalg^h_{\bH[1]} \bH[0]\to \bH[0]$ is given by the unique double functor to the terminal double category.

\bibliographystyle{alpha}
\bibliography{bib}

\end{document}